\UseRawInputEncoding
\documentclass[12pt, reqno]{amsart}
\usepackage[margin=1in]{geometry}
\usepackage{amssymb,latexsym,amsmath,amscd,amsfonts}
\usepackage{latexsym}
\usepackage[mathscr]{eucal}
\usepackage{bm}
\usepackage{mathptmx}

\newcommand{\hsp}{\hspace{5mm}}

\def\textmatrix#1&#2\\#3&#4\\{\bigl({#1 \atop #3}\ {#2 \atop #4}\bigr)}
\def\dispmatrix#1&#2\\#3&#4\\{\left({#1 \atop #3}\ {#2 \atop #4}\right)}
\newcommand{\beg}{\begin{equation}}
	\newcommand{\eeg}{\end{equation}}
\newcommand{\ben}{\begin{eqnarray*}}
	\newcommand{\een}{\end{eqnarray*}}

\newlength{\bibitemsep}
\newlength{\bibparskip}
\let\oldthebibliography\thebibliography
\renewcommand\thebibliography[1]{%
	\oldthebibliography{#1}%
	\setlength{\parskip}{\bibitemsep}%
	\setlength{\itemsep}{\bibparskip}%
}
\newtheorem{thm}{Theorem}[section]

\newtheorem{lem}[thm]{Lemma}

\numberwithin{equation}{section} 
\theoremstyle{definition}
\newtheorem{defn}[thm]{Definition}
\newtheorem{rem}[thm]{Remark}
\newtheorem{note}[thm]{Note}
\newtheorem{eg}[thm]{Example}

\newcommand{\HS}{\mathcal H}
\newcommand{\C}{\mathbb{C}}

\newcommand{\D}{\mathbb{D}}
\newcommand{\T}{\mathbb{T}}

\newcommand{\ov}{\overline}

\begin{document}
	\title[Minimal isometric dilations and operator models for the polydisc]
	{Minimal isometric dilations and operator models for the polydisc}
	
	\author[Pal and Sahasrabuddhe]{Sourav Pal and Prajakta Sahasrabuddhe}
	
	\address[Sourav Pal]{Mathematics Department, Indian Institute of Technology Bombay,
		Powai, Mumbai - 400076, India.} \email{souravpal@iitb.ac.in , souravmaths@gmail.com}
		
	\address[Prajakta Sahasrabuddhe]{Mathematics Department, Indian Institute of Technology Bombay,
		Powai, Mumbai - 400076, India.} \email{prajakta@math.iitb.ac.in , praju1093@gmail.com}

	
	
	\thanks{The first named author is supported by the Seed Grant of IIT Bombay, the CDPA and the `Core Research Grant' with Award No. CRG/2023/005223 of Science and Engineering Research Board (SERB), India. The second named author has been supported by the Ph.D Fellowship of Council of Scientific and Industrial Research (CSIR), India.}

	\begin{abstract}
		
 For commuting contractions $T_1,\dots ,T_n$ acting on a Hilbert space $\HS$ with $T=\prod_{i=1}^n T_i$, we find a necessary and sufficient condition such that $(T_1,\dots ,T_n)$ dilates to a commuting tuple of isometries $(V_1,\dots ,V_n)$ on the minimal isometric dilation space of $T$ with $V=\prod_{i=1}^nV_i$ being the minimal isometric dilation of $T$. This isometric dilation provides a commutant lifting of $(T_1, \dots , T_n)$ on the minimal isometric dilation space of $T$. We construct both Sch$\ddot{a}$ffer and Sz. Nagy-Foias type isometric dilations for $(T_1,\dots ,T_n)$ on the minimal dilation spaces of $T$. Also, a different dilation is constructed when the product $T$ is a $C._0$ contraction, that is ${T^*}^n \rightarrow 0$ as $n \rightarrow \infty$. As a consequence of these dilation theorems we obtain different functional models for $(T_1,\dots  ,T_n)$ in terms of multiplication operators on vectorial Hardy spaces. One notable fact about our models is that the multipliers are all analytic functions in one variable. The dilation when $T$ is a $C._0$ contraction leads to a conditional factorization of $T$. Several examples have been constructed.
  
	\end{abstract}
	
	
	\maketitle
	
	\tableofcontents
		
	\section{Introduction}
	\vspace{0.4cm}
	
\noindent We consider only bounded operators acting on complex Hilbert spaces. A contraction is an operator with norm not greater than $1$.

The aim of dilation, roughly speaking, is to realize a given tuple of commuting operators as a compression of an appropriate commuting tuple of normal operators. Let $(T_1, \dots , T_n)$ be a tuple of commuting contractions acting on a Hilbert space $\HS$. One would like to represent $(T_1, \dots , T_n)$ as a compression of an $n$-tuple of commuting unitaries or more precisely as a compression of an $n$-tuple of commuting isometries, because, every such tuple of commuting isometries extends naturally to a commuting tuple of unitaries. A commuting tuple of isometries $(V_1, \dots , V_n)$ acting on a Hilbert space $\mathcal K$ is said to be an \textit{isometric dilation} of $(T_1, \dots , T_n)$ if $\HS$ can be identified as a closed linear subspace of $\mathcal K$, i.e. $\HS \subseteq \mathcal K$ and for any non-negative integers $k_1, \dots , k_n$
\[
T_1^{k_1} \dots T_n^{k_n}=P_{\HS} (V_1^{k_1}\dots V_n^{k_n})|_{\HS},
\]	
where $P_{\HS}:\mathcal K \rightarrow \HS$ is the orthogonal projection. Moreover, such an isometric dilation is called minimal if
\[
\mathcal K = \ov{Span}\; \{ V_1^{k_1}\dots V_n^{k_n}h\;:\; h\in \HS , \; k_1, \dots , k_n \in \mathbb N \cup \{0 \} \}.
\]	
If $(V_1, \dots , V_n)$ dilates $(T_1, \dots , T_n)$, then each $T_i$ is a compression of $V_i$, that is $T_i=P_{\HS}V_i|_{\HS}$. It is well-known that a contraction admits an isometric dilation (Sz. Nagy, \cite{Nagy 1}) and that a pair of commuting contractions always dilates to a pair of commuting isometries (Ando, \cite{Ando}), though a triple of commuting contractions may or may not dilate to a triple of commuting isometries (Parrott, \cite{Par}). In other words, rational dilation succeeds on the closed unit disk $\ov{\mathbb D}$ and on the closed bidisc $\ov{\mathbb D}^2$ and fails on the closed polydisc $\ov{\D}^n$ when $n \geq 3$. Since a commuting tuple of contractions $(T_1, \dots, T_n)$ does not dilate unconditionally whenever $n\geq 3$, efforts have been made to find classes of contractions that dilate under certain conditions and some remarkable works have been witnessed, e.g. Agler \cite{Agler 1}, Arveson \cite{Arv 3}, Ball, Li, Timotin, Trent \cite{Bal:Tim:Tre}, Ball, Trent, Vinnikov \cite{Bal:Tre:Vin}, Bhat, Bhattacharyya, Dey \cite{Bha:Bhat:Dey}, Binding, Farenick, Li \cite{Bin:Far:chi}, Brehmer \cite{Bre}, Crabb, Davie \cite{Cra:Dav}, Curto, Vasilescu \cite{Cur:Vas 1, Cur:Vas 2}, Dey \cite{Dey}, Grinshpan, Kaliuzhnyi-Verbovetki, Vinnikov, Woerdeman \cite{G:K:V:V:W}, Muller, Vasilescu \cite{Mul:Vas}, Popescu \cite{Pop 5} and many others. See the references therein and also see Section 2 for further details.

\medskip

In this article, we consider the minimal isometric dilation space $\mathcal K$ (which is always unique upto unitary) of the product $T=\prod_{i=1}^n T_i$ of a tuple commuting contractions $(T_1, \dots , T_n)$ acting on $\HS$. We find a necessary and sufficient condition such that $(T_1, \dots , T_n)$ dilates to a commuting isometric tuple $(V_1, \dots , V_n)$ on $\mathcal K$ with the product $V=\prod_{i=1}^n V_i$ being the minimal isometric dilation of $T=\prod_{i=1}^n T_i$. Note that the space $\mathcal K$ is unique in the sense that any two minimal isometric dilation spaces of the product $T$ are unitarily equivalent. This is one of the main results in this article and is stated below.

\begin{thm} \label{main-intro}

	Let $T_1,\ldots,T_n\in \mathcal{B}(\mathcal{H})$ be commuting contractions and let $T={ \prod_{i=1}^n T_i}$. 
	\begin{itemize}
		\item[(a)] If $\mathcal{K}$ is the minimal isometric dilation space of $T$, then $(T_1,\ldots ,T_n)$ possesses an isometric dilation $(V_1,\ldots,V_{n})$ on $\mathcal{K}$ with $V=\Pi_{i=1}^nV_i$ being the minimal isometric dilation of $T$ if and only if there are unique orthogonal projections $P_1,\ldots ,P_n$ and unique commuting unitaries $U_1,\ldots ,U_n$ in $\mathcal{B}(\mathcal{D}_T)$ with $\prod_{i=1}^n U_i=I_{\mathcal{D}_{T}}$ such that the following conditions are satisfied for each $i=1, \dots, n$:
		\begin{enumerate} 
			\item $D_TT_i=P_i^{\perp}U_i^*D_T+P_iU_i^*D_TT$ , 
			\item  $P_i^{\perp}U_i^*P_j^{\perp}U_j^*=P_j^{\perp}U_j^*P_i^{\perp}U_i^*$ ,
			\item $U_iP_iU_jP_j=U_jP_jU_iP_i$ ,
			\item $D_TU_iP_iU_i^*D_T=D_{T_i}^2$ ,
			\item  $P_1+U_1^*P_2U_1+U_1^*U_2^*P_3U_2U_1+\ldots +U_1^*U_2^*\ldots U_{n-1}^*P_nU_{n-1}\ldots U_2U_1 =I_{\mathcal{D}_T}$.
		\end{enumerate}
		\vspace{2mm}

		\item[(b)] Such an isometric dilation is minimal and unique in the sense that if $(W_1, \dots , W_n)$ on $\mathcal K_1$ and $(Y_1, \dots , Y_n)$ on $\mathcal K_2$ are two isometric dilations of $(T_1, \dots , T_n)$ such that $W=\prod_{i=1}^nW_i$ and $Y=\prod_{i=1}^nY_i$ are minimal isometric dilations of $T$ on $\mathcal K_1$ and $\mathcal K_2$ respectively, then there is a unitary $\widetilde{U}:\mathcal K_1 \rightarrow \mathcal K_2$ such that $(W_1, \dots , W_n)=(\widetilde{U}^*Y_1\widetilde{U}, \dots , \widetilde{U}^*Y_n\widetilde{U})$.
	\end{itemize}			
	
\end{thm}
This is Theorem \ref{main} in this paper and will be proved in Section \ref{sec:03}. We show an explicit construction of a Sch$\ddot{a}$ffer type minimal isometric dilation for $(T_1, \dots , T_n)$ on the space $\mathcal K= \HS \oplus l^2(\mathcal D_T)$, where $\mathcal D_T=\ov{Ran}(I-T^*T)^{\frac{1}{2}}$ (see Theorem \ref{main}). We also show in Theorem \ref{coromain} that such a dilation can be constructed with the conditions $(1)-(4)$ of Theorem \ref{main-intro} only, though we do not have an exact converse part then. A special emphasis is given to the case when the product $T$ is a $C._0$ contraction, i.e. ${T^*}^n\rightarrow 0$ strongly as $n\rightarrow \infty$. We show in Theorem \ref{puredil} that an analogue	of Theorem \ref{main-intro} can be achieved for the $C._0$ case with a weaker hypothesis. We explicitly construct an isometric dilation in this case too. This leads to a functional model and a factorization of a $C._0$ contraction. A notable fact about this model is that the multipliers involved here are linear analytic functions in one variable. In Theorem \ref{Nagy isodil}, another main result of this paper, we construct explicitly a similar isometric dilation for $(T_1, \dots , T_n)$ on the Sz. Nagy-Foias minimal isometric dilation space of $T$. In Section \ref{Examples}, we provide several examples describing different classes of commuting contractions that dilate to commuting isometries conditionally. There we show that our classes of commuting contractions admitting isometric dilations are not properly contained in any of the previously determined classes in the literature. Also, none of such classes from the literature is a proper subclass of our classes, however there are intersections. Finally, we present a model theory for a class of commuting contractions in Section \ref{model-theory}.

		\vspace{0.1cm}
		
		\section{A brief history of dilation on the polydisc}
		
		\vspace{0.2cm}
		
\noindent An isometric dilation $(V_1, \dots , V_n)$ of $(T_1, \dots , T_n)$ naturally extends to a tuple of commuting unitaries $(U_1, \dots , U_n)$ and consequently $(U_1, \dots , U_n)$ becomes a unitary dilation of $(T_1, \dots , T_n)$. Since $(U_1, \dots , U_n)$ is a tuple of commuting unitaries having its Taylor joint spectrum on the $n$-torus $\T^n$ which is the distinguished boundary of the closed polydisc $\ov{\D}^n$, following Arveson's terminology (see \cite{Arv 2}) we say that $\ov{\D}^n$ is a complete spectral set for $(T_1, \dots , T_n)$. So, it follows that $\ov{\D}^n$ is a spectrral set for $(T_1, \dots , T_n)$. Thus, the $n$-tuples of commuting contractions that dilate to commuting isometries or unitaries must have $\ov{\D}^n$ as a spectral set. In \cite{Hal}, Halmos constructed a unitary $U$ on a certain larger space for a contraction $T$ acting on a Hilbert space $\mathcal H$ such that $T=P_{\HS}U|_{\HS}$, which is to say that $T$ is a compression of a unitary $U$. Existence of an isometry satisfying such a compression relation was proved before it by Julia, e.g. see \cite{Jul 1}, \cite{Jul 2}, \cite{Jul 3}. The unitary dilation of Halmos was missing the compression-vs-dilation frame for positive integral powers of $U$ and $T$. Later, Sz. Nagy resolved this issue in \cite{Nagy 2} with an innovative idea, where he proved that there is a Hilbert space $\mathcal{K}$ containing $\mathcal{H}$ and a unitary $U$ on $\mathcal K$ such that $T^n=P_{\HS}U^n|_{\HS}$ for any non-negative integer $n$. This is well-known as Sz. Nagy unitary dilation of a contraction. A few years after Sz. Nagy's famous discovery, Douglas \cite{Dou 3} and J. J. Sch$\ddot{a}$ffer \cite{Sha} produced distinct and explicit constructions of such a unitary dilation (of a contraction). The pioneering works of Sz.-Nagy, Douglas and Sch$\ddot{a}$ffer were further generalized by Ando to a pair of commuting contractions. Indeed, in \cite{Ando} Ando constructed an isometric dilation $(V_1,V_2)$ for a pair of commuting contractions $(T_1,T_2)$. Success of dilation for a pair of commuting contractions led to the natural question, whether an arbitrary $n$-tuple of commuting contractions dilates to some $n$-tuple of commuting isometries or unitaries for $n\geq 3$. This question was answered negatively by S. Parrott in \cite{Par} via a counter example. One way of realizing the impact of this dilation result is the celebrated von-Neumann inequality.
\begin{thm} [\cite{Von}]
	Let $T$ be a contraction on some Hilbert space $\mathcal{H}$. Then for every polynomial $p\in \mathbb{C}[z]$,
	\[ \|p(T\| \leq \sup_{|z|\leq 1}|p(z)|.\] 
\end{thm}
It was observed that the existence of a unitary dilation is sufficient for a commuting tuple of contractions to satisfy von-Neumann inequality. Using this principle, Crabb and Davie in \cite{Cra:Dav} and Varopoulos in \cite{Var} produced examples of a triple of commuting contractions which do not satisfy Von-Neumann inequality and hence do not admit a unitary dilation. These examples spurred a lot of mathematicians to look into the von-Neumann inequality for commuting contractions at least on the finite dimensional spaces, \cite{Cho:Li}, \cite{Dru 1}, \cite{Dru 2}, \cite{Hol}, \cite{Kni}. In their seminal paper \cite{Agl:Mac}, Agler and McCarthy proved a sharper version of von-Neumann inequality for a pair of commuting and strictly contractive matrices. In \cite{Das:Sar}, Das and Sarkar presented a new proof to the result of Agler and McCarthy with a refinement of the class of matrices. The impact of Ando's dilation is eminent even in the 20th century. In \cite{Bag:Bhatt}, Bagchi, Bhattacharyya, Misra have presented an elementary proof of Ando’s theorem in a $C^*$-algebraic setting, within a restricted class of homomorphisms modeled after Parrott’s example. In \cite{Sau}, Sau gave new proofs to And$\hat{o}$’s dilation theorem with Sch$\ddot{a}$ffer and Douglas-type constructions.

\smallskip
 
 In \cite{Bre}, Brehmer introduced the concept of regular unitary dilation and systematically studied the existence of such dilation. For any $\alpha\in \mathbb{Z}^n$ let $ \alpha_{-}=(-\min\{o,\alpha_1\},\ldots , -\min\{0, \alpha_n\})$ and $\alpha_{+}=(\max\{ 0,\alpha_1\},\ldots , \max\{0,\alpha_n\})$. For a given commuting $n$-tuple of contractions $(T_1,\ldots,T_n)$ and a tuple of positive integers $m=(m_1,\ldots ,m_n)$, the following notation is used in the literature: ${T}^m:=\Pi_{i=1}^nT_i^{m_i}$.
   
\begin{defn}[Section 9, \cite{Nagy}]
	A commuting $n$-tuple of unitaries $U=(U_1,\ldots ,U_n)$ on a Hilbert space $\mathcal{K}$ is said to be a regular unitary dilation of a commuting $n$-tuple of contractions $T=(T_1,\ldots,T_n)$ on $\mathcal{H}\subseteq \mathcal{K}$, if for any $\alpha\in \mathbb{Z}^n$, \[T^{*\alpha_{-}}T^{\alpha_{+}}=P_{\mathcal{H}}U^{*\alpha_{-}}U^{\alpha_{+}}|_{\mathcal{H}}.  \] 
\end{defn}
Brehmer proved in \cite{Bre} that a tuple of commuting contractions, if admits a regular unitary (or isometric) dilation, can be completely characterized by some positivity conditions, which is known as Brehmer's positivity. A tuple $T$ is said to satisfy Brehmer's positivity condition if
\begin{equation} \label{lit:eqn01}
\sum_{F\subseteq G}(-1)^{|F|}T_{F}^*T_F\geq 0 
\end{equation}
for all $G\subseteq\{ 1,\ldots n \}$. It follows from the definition that the existence of a regular unitary dilation implies the existence of a unitary dilation for a commuting tuple of contractions. The study of Brehmer was further continued by Halperin in \cite{Halp 1} and \cite{Halp 2}. The positivity condition introduced by Brehmer attracted considerable attentions, e.g see the novel works due to Agler \cite{Agler 1}, Curto and Vasilescu  \cite{Cur:Vas 1},\cite{Cur:Vas 2}. Indeed, Curto and Vasilescu generalized the original theorem of Brehmer together with Agler's results on hypercontractivity by general model theory for multi-operators which satisfy certain positivity conditions. An alternative approach to the results due to Agler, Curto and Vasilescu was provided by Timotin in \cite{Tim}. Timotin's approach had thrown some new lights on the geometric and combinatorial parts of the model theory of Agler, Curto and Vasilescu. In \cite{Bin:Far:chi}, Binding, Farenick and Li proved that for every $m$-tuple of operators on a Hilbert space, one can simultaneously dilate them to normal operators on the same Hilbert space such that the dilating operators have finite spectrums. On the other hand, there are nontrivial results on dilation of a contractive but not necessarily commuting tuples. In \cite{Dav}, Davis started studying such tuples and then Bunce \cite{Bun} and Frazho \cite{Fra} provided a wider and concrete form to this analysis. An extensive research in the direction of non-commuting dilation has been carried out by Popescu in \cite{Pop 1}-\cite{Pop 5} and also in collaboration with Arias in \cite{Ari:Pop 1}, \cite{Ari:Pop 2}. In \cite{Bha:Bhat:Dey}, Bhat, Bhattacharyya and Dey proved that for a commuting contractive tuple the standard commuting dilation is the maximal commuting dilation sitting inside the standard non-commuting dilation. 

In \cite{Arv 3}, Arveson considered a $d-$tuple $(T_1,\ldots ,T_d)$ of mutually commuting operators acting on a Hilbert space $\mathcal{H}$ such that
\[
\|T_1h_1+\ldots +T_dh_d \|^2\leq \|h_1\|^2+\ldots +\|h_d\|^2.
\]
He showed many of the operator-theoretic aspects of function theory of the unit disk generalize to that of the unit ball $B_d$ in complex $d-$space, including von-Neumann inequality and the model theory of contractions. Apart from this, the notable works due to Athavale \cite{Ath 1}-\cite{Ath 3}, Druy \cite{Dru 1}, Vasilescu \cite{Vas 1} were among the early contributors to the multi-parameter operator theory on the unit ball in $\mathbb{C}^n$. In \cite{Mul:Vas}, Muller and Vasilescu analyzed some positivity conditions for commuting
multi-operators which ensured the unitary equivalence of these objects to some standard models consisting of backwards multishifts. They considered spherical dilation of a tuple of commuting contractions $(T_1,\ldots T_d)$ on $\HS$. Such a tuple dilates to a tuple of commuting normal operators $(N_1,\ldots ,N_d)$ on $\mathcal{K} \supseteq \mathcal H$ satisfying
\[
N_1^*N_1+\cdots +N_d^*N_d=I_{\mathcal{K}}.
\]
In \cite{Mul:Vas}, Muller and Vasilescu gave a necessory and sufficient condition for a commuting multioperator to have spherical dilation in terms of positivity of certain operator polynomials involving $T$ and $T^*$. The dilation results of Sz. Nagy \cite{Nagy 1} for contractions and Agler \cite{Agler 2} for $m-$hypercontractions follow as a special case of the result due to Muller and Vasilescu. It is evident that unlike unitary dilation, the tuple of contractions that admit regular dilations can be completely characterized by Brehmer's positivity conditions \cite{Bre}. So, this means that Curto and Vasilescu in \cite{Cur:Vas 2} have found a bigger class of contractive tuples which admit commuting unitary dilations. Later Grinshpan, Kaliuzhnyi, Verbovetskyi, Vinnikov and Woerdeman \cite{G:K:V:V:W} extended this result to a bigger class which was denoted by $\mathcal{P}^d_{p,q}$. Recently Barik, Das, Haria and Sarkar \cite{B:D:H:S} introduced even a larger class of commuting contractions, denoted by $\mathcal{T}^n_{p,q}(\mathcal{H})$, which dilate to commuting isometries. Also, Barik and Das established a von-Neumann inequality for a tuple of commuting contractions belonging to $\mathcal{B}^n_{p,q}$. In the expository essay \cite{Lev:Sha}, Levy and Shalit discussed a finite dimensional approach to dilation theory and have answered to some extent how much of the dilation theory can work out within the realm of linear algebra. Also, an interested reader is referred to \cite{Mac:Sha 1} due to McCarthy and Shalit. In \cite{Sto:Sza}, Stochel and Szafraniec proposed a test for a commutative family of operators to have a unitary power dilation. For a detailed study of dilation theory, an interested reader is also referred to the nice survey articles by Bhattacharyya \cite{Bha 3} and Shalit \cite{Sha}. 

\vspace{0.1cm}

\section{Sch$\ddot{a}$ffer type minimal isometric dilation} \label{sec:03}
	
\vspace{0.3cm}
	
\noindent Let us recall a few notations and terminologies from the literature. For a contraction $T$ on a Hilbert space $\HS$, the \textit{defect operator} of $T$ is the unique positive square root of $I-T^*T$ and it is denoted by $D_T$. Also, the closure of the range of $D_T$ is denoted by $\mathcal D_T$, i.e. $\mathcal D_T=\overline{Ran} \, D_T$. A contraction $T\in \mathcal B(\HS)$ is called \textit{completely non-unitary} or simply \textit{c.n.u.} if there is no non-zero subspace $\HS_1$ of $\HS$ that reduces $T$ and on which $T$ acts as a unitary. The classical $L^2$ space consists of complex-valued functions defined on the unit circle $\T$ that are square integrable with respect to the Lebesgue measure on $\T$. A canonical basis for $L^2$ is $\{e^{in\theta}\,:\, n\in \mathbb Z \}$ and the closed subspace of $L^2$ generated by the basis $\{ e^{i n\theta}\,:\, n=0,1,2, \dots \}$ is denoted by $\widetilde{H}^2$. For any Hilbert space $E$, the space $L^2(E)$ is defined similarly as $L^2$, the only difference is that the functions in $L^2(E)$ are $E$-valued. It is well-known that the Hilbert spaces $L^2(E)$ and $L^2 \otimes E$ are unitarily equivalent. Under this unitary equivalence, the replica of $\widetilde{H}^2 \otimes E$ in $L^2(E)$ is denoted by $\widetilde{H}^2(E)$. A \textit{multiplication operator} $M_{\phi}$ on $L^2(E)$, where $\phi(z)$ is an essentially bounded function from $\T$ to $E$, i.e. $\phi \in L^{\infty}(E)$, is defined by $M_{\phi}f(z)=\phi(z)f(z)$. For any $\phi \in L^{\infty}(E)$, the \textit{Toeplitz} operator $T_{\phi}$ on $\widetilde{H}^2(E)$ is defined by $T_{\phi}f(z)=P \phi(z)f(z)$, where $P:L^2(E) \rightarrow \widetilde{H}^2(E)$ is the orthogonal projection. For any Hilbert space $E$, the \textit{Hardy space} $H^2(E)$ consists of analytic functions from the unit disk $\D$ to $E$ with square summable coefficients in its power series, i.e.
\[
H^2(E) =\{f:\D \rightarrow E \,:\, f(z)=\sum_{i=0}^{\infty} \, a_nz^n\,, \, \, a_n \in E \, \text{ for all } n\in \mathbb N \cup \{0\} \, \& \, \sum_{i=0}^{\infty}\, \|a_n\|^2 < \infty  \}.
\]
The Hilbert spaces $\widetilde{H}^2(E)$ and $H^2(E)$ are unitarily equivalent. A multiplication operator $M_{\phi}$ on $H^2(E)$, where $\phi(z)$ is an analytic multiplier, is defined by $M_{\phi}f(z)=\phi(z)f(z)$.

\smallskip

 To explain the results of this Section, we begin with the Berger-Coburn-Lebow model (or, simply the BCL model) for commuting isometries which will be used in sequel.

\begin{thm}[Berger-Coburn-Lebow, \cite{Ber}] \label{BCL}
	Let $V_1, \dots , V_n$ be commuting isometries on $\HS$ such that $V=\prod_{i=1}^n V_i$ is a pure isometry. Then, there exist projections $P_1, \dots , P_n$ and unitaries $U_1, \dots , U_n$ in $\mathcal B(\mathcal D_{V^*})$ such that
\[
(V_1, \dots , V_n , V) \equiv (T_{P_1^{\perp}U_1+ zP_1U_1 }, \dots , T_{P_n^{\perp}U_n+zP_nU_n }, T_z)  \;\; \text{ on } \; \; H^2(\mathcal D_{V^*}).
\] 

\end{thm}
 Later, Bercovici, Douglas and Foias found a refined operator model for commuting c.n.u. isometries in \cite{Berc:Dou:Foi}. They introduced the notion of model $n$-iometries. A \textit{model} $n$-\textit{isometry} is a tuple of commuting $n$-isometries $(V_1, \dots , V_n)$ such that each $V_i$ is a multiplication operator of the form $M_{U_iP_i^{\perp}+zU_iP_i}$ and $\prod_{i=1}^n V_i=M_z$, where $P_1, \dots , P_n$ are orthogonal projections and $U_1, \dots , U_n$ are unitaries acting on a Hilbert space $\HS$. The following characterization theorem for model $n$-isometries is nothing but a variant of the model due to Bercovici, Douglas and Foias and a proof follows from Lemma 2.2 in \cite{Berc:Dou:Foi} and the discussion bellow it. This will be used in sequel.

		\begin{thm}[Bercovici, Douglas and Foias, \cite{Berc:Dou:Foi}] \label{BDF lemma O}
			Let $U_1,\ldots ,U_n$ be unitaries on Hilbert space $\mathcal{H}$  and $P_1,\ldots ,P_n$ be orthogonal projections on $\mathcal{H}$. For each $1\leq i \leq n$,  let $V_i=M_{U_iP_i^{\perp}+zU_iP_i}$. Then $(V_1,\ldots ,V_n)$ defines a commuting $n$-tuple of isometries with $\Pi_{i=1}^nV_i=M_z$ if and only if the following conditions are satisfied. \begin{enumerate}
				\item $U_iU_j=U_jU_i$ for all $1\leq i <j \leq n$,
				\item $U_1\ldots U_n=I_{\mathcal{H}}$,
				\item $P_j+U_j^*P_iU_j=P_i+U_i^*P_jU_i \leq I_{\mathcal{H}}$, for all $i\neq j$ and 
				\item  $P_1+U_1^*P_2U_1+U_1^*U_2^*P_3U_2U_1+\ldots +U_1^*U_2^*\ldots U_{n-1}^*P_nU_{n-1}\ldots U_2U_1 =I_{\mathcal{H}}$.
				
			\end{enumerate} 
		\end{thm}
		The following result is a corollary of Theorem \ref{BDF lemma O}. In other words, this can be treated as a variant of Theorem \ref{BDF lemma O}. We state it here so that we can directly apply it later.
		
		\begin{thm} \label{BDF lemma}
			Let $U_1,\ldots ,U_n$ be unitaries on Hilbert space $\mathcal{H}$  and $P_1,\ldots ,P_n$ be orthogonal projections on $\mathcal{H}$. For each $1\leq i \leq n$,  let $V_i=M_{P_i^{\perp}U_i^*+zP_iU_i^*}$. Then $(V_1,\ldots ,V_n)$ defines a commuting $n$-tuple of isometries with $\Pi_{i=1}^nV_i=M_z$ if and only if the following conditions are satisfied.
			\begin{enumerate}
				\item $U_iU_j=U_jU_i$ for all $1\leq i <j \leq n$,
				\item $U_1\ldots U_n=I_{\mathcal{H}}$,
				\item $P_j+U_j^*P_iU_j=P_i+U_i^*P_jU_i \leq I_{\mathcal{H}}$, for all $i\neq j$ and 
				\item  $P_1+U_1^*P_2U_1+U_1^*U_2^*P_3U_2U_1+\ldots +U_1^*U_2^*\ldots U_{n-1}^*P_nU_{n-1}\ldots U_2U_1 =I_{\mathcal{H}}$.
\end{enumerate} 
		\end{thm}

Now we present a twisted version of Lemma 2.2 in \cite{Berc:Dou:Foi}. This will be used in the proof of the main theorem. For the sake of completeness we give a proof here and it goes along with similar arguments as in the proof of Lemma 2.2 in \cite{Berc:Dou:Foi}.
  
		\begin{lem}\label{BDF lemma 1}
			Consider unitary operators $U,U_1,U_2$ and orthogonal projections $P,P_1$ and $P_2$ on Hilbert space $\mathcal{H}$. If $V_{U,P},V_{U_1,P_1}$ and $V_{U_2,P_2}$ on $H^2(\mathcal{H})$ are defined as $V_{U,P}=M_{P^{\perp}U^*+zPU^*}$, $V_{U_1,P_1}=M_{P_1^{\perp}U_1^*+zP_1U_1^*}$ and $V_{U_2,P_2}=M_{P_2^{\perp}U_2^*+zP_2U_2^*}$, then the following are equivalent: 
			\begin{enumerate}
				\item[(i)] $ V_{U,P}=V_{U_1,P_1}V_{U_2,P_2} $,
				\item[(ii)] $U=U_2U_1  $ and $P=P_1+U_1^*P_2U_1$.
			\end{enumerate}
		\end{lem}
		\begin{proof}
			We prove only $(i)\implies (ii)$, the proof of $(ii) \implies (i)$ follows trivially. From $(i)$ we have $ V_{U,P}=V_{U_1,P_1}V_{U_2,P_2} $ and thus 
			\begin{align*} 
				&P^{\perp}U^*+zPU^*=(P_1^{\perp}U_1^*+zP_1U_1^*)(P_2^{\perp}U_2^*+zP_2U_2^*)\\
				\text{i.e.}\quad  &P^{\perp}U^*+zPU^* =P_1^{\perp}U_1^*P_2^{\perp}U_2^*+z(P_1U_1^*P_2^{\perp}U_2^*+P_1^{\perp}U_1^*P_2U_2^*)+z^2P_1U_1^*P_2U_2^*.
			\end{align*}
			Hence we have
			\begin{enumerate}
				\item $P^{\perp}U^*=P_1^{\perp}U_1^*P_2^{\perp}U_2^*$ ;
				\item $PU^*=P_1U_1^*P_2^{\perp}U_2^*+P_1^{\perp}U_1^*P_2U_2^* $ ;
				\item $P_1U_1^*P_2U_2^*=0$.
			\end{enumerate}
			From $(2)$ and $(3)$ by substituting $P_i^{\perp}=I-P_i$, we obtain, \begin{equation}\label{PU*}
				PU^*=U_1^*P_2U_2^*+P_1U_1^*U_2^*
			\end{equation}
			From $(1)$ and $(3)$ by substituting $P_i^{\perp}=I-P_i$, we obtain, 
			\begin{equation}\label{U*-PU*}
				U^*-PU^*=U_1^*U_2^*-U_1^*P_2U_2^*-P_1U_1^*U_2^*
			\end{equation}
			Hence, \eqref{PU*} and \eqref{U*-PU*} give us $U=U_2U_1$. Again, multiplying \eqref{PU*} from right by $U_2U_1$ and substituting $U=U_2U_1$ we obtain $P=P_1+U_1^*P_2U_1$. The proof is now complete. 
			  
		\end{proof}

We now present a Sch$\ddot{a}$ffer-type minimal isometric dilation for a tuple of commuting contractions and this is one of the main results of this paper. However, the proof of this theorem is going to be lengthy and so we will split the proof into several parts. We request the readers to kindly bear with us for once.		
		
	\begin{thm}\label{main}
	Let $T_1,\ldots,T_n\in \mathcal{B}(\mathcal{H})$ be commuting contractions and let $T={ \prod_{i=1}^n T_i}$. 
	\begin{itemize}
		\item[(a)] If $\mathcal{K}$ is the minimal isometric dilation space of $T$, then $(T_1,\ldots ,T_n)$ possesses an isometric dilation $(V_1,\ldots,V_{n})$ on $\mathcal{K}$ with $V=\Pi_{i=1}^nV_i$ being the minimal isometric dilation of $T$ if and only if there are unique orthogonal projections $P_1,\ldots ,P_n$ and unique commuting unitaries $U_1,\ldots ,U_n$ in $\mathcal{B}(\mathcal{D}_T)$ with $\prod_{i=1}^n U_i=I_{\mathcal{D}_{T}}$ such that the following conditions are satisfied for each $i=1, \dots, n$:
		\begin{enumerate} 
			\item $D_TT_i=P_i^{\perp}U_i^*D_T+P_iU_i^*D_TT$ , 
			\item  $P_i^{\perp}U_i^*P_j^{\perp}U_j^*=P_j^{\perp}U_j^*P_i^{\perp}U_i^*$ ,
			\item $U_iP_iU_jP_j=U_jP_jU_iP_i$ ,
			\item $D_TU_iP_iU_i^*D_T=D_{T_i}^2$ ,
			\item  $P_1+U_1^*P_2U_1+U_1^*U_2^*P_3U_2U_1+\ldots +U_1^*U_2^*\ldots U_{n-1}^*P_nU_{n-1}\ldots U_2U_1 =I_{\mathcal{D}_T}$.
		\end{enumerate}
		\vspace{2mm}

		\item[(b)] Such an isometric dilation is minimal and unique in the sense that if $(W_1, \dots , W_n)$ on $\mathcal K_1$ and $(Y_1, \dots , Y_n)$ on $\mathcal K_2$ are two isometric dilations of $(T_1, \dots , T_n)$ such that $W=\prod_{i=1}^nW_i$ and $Y=\prod_{i=1}^nY_i$ are minimal isometric dilations of $T$ on $\mathcal K_1$ and $\mathcal K_2$ respectively, then there is a unitary $\widetilde{U}:\mathcal K_1 \rightarrow \mathcal K_2$ such that $(W_1, \dots , W_n)=(\widetilde{U}^*Y_1\widetilde{U}, \dots , \widetilde{U}^*Y_n\widetilde{U})$.
	\end{itemize}			
	
\end{thm}

		\begin{proof}
		   
	\textbf{(a).} \textbf{(\textit{The $\Leftarrow$ part}).} Suppose there are projections $P_1,\ldots,P_n$ and commuting unitaries $U_1,\ldots ,U_n$ in $\mathcal{B}(\mathcal{D}_T)$ with $\prod_{i=1}^n U_i=I$ satisfying the operator identities $(1)-(5)$ for $1\leq i \leq n$. We first show that conditions $(1)-(4)$ guaranttee the existence of an isometric dilation of $(T_1, \dots , T_n)$. In fact, we shall construct a co-isometric extension of $(T_1^*, \dots , T_n^*)$. It is well-known from Sz.-Nagy-Foias theory (see \cite{Nagy}) that any two minimal isometric dilations of a contraction are unitarily equivalent. Thus, without loss of generality we consider the Sch$\ddot{a}$ffer's minimal isometric dilation space $\mathcal K_0$ of $T$, where $\mathcal{K}_0=\mathcal{H}\oplus l^2(\mathcal{D}_{T})=\HS \oplus \mathcal D_T \oplus \mathcal D_T \oplus \ldots$ and construct an isometric dilation on $\mathcal K_0$ for $(T_1, \dots , T_n)$. Define $V_i$ on $\mathcal{K}_0$ as follows:
		  \begin{equation} \label{eqn:isom-01}
		  V_i=\begin{bmatrix}
		  	T_i&0&0&0&\ldots \\
		  	P_iU_i^*D_T&P_i^{\perp}U_i^*&0&0&\ldots \\
		  	0&P_iU_i^*&P_i^{\perp}U_i^*&0&\ldots \\
		  	0&0&P_iU_i^*&P_i^{\perp}U_i^*&\ldots \\
		  	\ldots & \ldots & \ldots & \ldots & \ldots
		  \end{bmatrix}, \qquad 1\leq i \leq n. 
		  \end{equation}
		  It is evident from the block-matrix form that $V_i^*|_{\HS} =T_i^*$ for each $i=1, \ldots , n$.\\
		  
		  \noindent \textbf{\textit{Step 1.}} First we prove that $(V_1,\ldots ,V_n)$ is a commuting tuple. For each $i,j$ we have that	
		  \[
			V_iV_j=
			\begin{bmatrix}
					T_iT_j&0&0&\ldots\\
					P_iU_i^*D_TT_j+P_i^{\perp}U_i^*P_jU_j^*D_T & P_i^{\perp}U_i^*P_j^{\perp}U_j^* &0&\ldots \\					P_iU_i^*P_jU_j^*D_T&P_iU_i^*P_j^{\perp}U_j^*+P_i^{\perp}U_i^*P_jU_j^*&P_i^{\perp}U_i^*P_{j}^{\perp}U_j^*&\ldots \\					0&P_iU_i^*P_jU_j^*&P_iU_i^*P_j^{\perp}U_j^*+P_i^{\perp}U_i^*P_jU_j^*& \ldots \\
					\ldots & \ldots & \ldots & \ddots
				\end{bmatrix}
			\]
			and
			\[
			  V_jV_i=
				\begin{bmatrix}
					T_jT_i&0&0&\ldots\\
					P_jU_j^*D_TT_i+P_j^{\perp}U_j^*P_iU_i^*D_T&P_j^{\perp}U_j^*P_i^{\perp}U_i^*&0&\ldots \\					P_jU_j^*P_iU_i^*D_T&P_jU_j^*P_i^{\perp}U_i^*+P_j^{\perp}U_j^*P_iU_i^*&P_j^{\perp}U_j^*P_{i}^{\perp}U_i^*&\ldots \\					0&P_jU_j^*P_iU_i^*&P_jU_j^*P_i^{\perp}U_i^*+P_j^{\perp}U_j^*P_iU_i^*& \ldots \\
\ldots&\ldots&\ldots &\ddots
\end{bmatrix}.
\]
Using the commutativity of $T_i,T_j$ and $U_i,U_j$ and applying condition-(3) of the theorem, we obtain by simplifying the condition-(2), i.e. $
		 	P_i^{\perp}U_i^*P_j^{\perp}U_j^*=P_j^{\perp}U_j^*P_i^{\perp}U_i^* $,
		  that \begin{equation}\label{identity 1} P_iU_i^*U_j^*+U_i^*P_jU_j^*=P_jU_j^*U_i^*+U_j^*P_iU_i^*. 
	  \end{equation}
	   Now we show that 
		  \begin{equation}\label{identity 2}
		  	P_jU_j^*P_i^{\perp}U_i^*+P_j^{\perp}U_j^*P_iU_i^*=P_iU_i^*P_j^{\perp}U_j^*+P_i^{\perp}U_i^*P_jU_j^*
		  \end{equation}
		  and
		  \begin{equation} \label{eqn:new-02}
		  P_jU_j^*D_TT_i+P_j^{\perp}U_j^*P_iU_i^*D_T=P_iU_i^*D_TT_j+P_i^{\perp}U_i^*P_jU_j^*D_T.
		  \end{equation}
		  
For showing (\ref{identity 2}), we see that
	  \begin{align*}
	  	P_jU_j^*P_i^{\perp}U_i^*+P_j^{\perp}U_j^*P_iU_i^*
	  &	= P_jU_j^*(U_i^*-P_iU_i^*)+(U_j^*-P_jU_j^*)P_iU_i^*\\
	  &	= P_jU_j^*U_i^*-P_jU_j^*P_iU_i^*+U_j^*P_iU_i^*-P_jU_j^*P_iU_i^*\\
	  &	= P_iU_i^*U_j^*-P_iU_i^*P_jU_j^*+U_i^*P_jU_j^*-P_iU_i^*P_jU_j^*\\
	  &	= P_iU_i^*P_j^{\perp}U_j^*+P_i^{\perp}U_i^*P_jU_j^*.
	  \end{align*}
 Note that the second last equality follows from equation (\ref{identity 1}) and condition-(3) of Theorem \ref{main}.

\smallskip

	To prove (\ref{eqn:new-02}), we use a few conditions of Theorem \ref{main} here. We have
\begin{align*}
P_jU_j^*D_TT_i+P_j^{\perp}U_j^*P_iU_i^*D_T
				& = P_jU_j^*(P_i^{\perp}U_i^*D_T+P_iU_i^*D_TT)+P_j^{\perp}U_j^*P_iU_i^*D_T \;\;\; [\text{by condition-(1)}] \\
				& = (P_jU_j^*P_i^{\perp}U_i^*+P_j^{\perp}U_j^*P_iU_i^*)D_T+P_jU_j^*P_iU_i^*D_TT\\
				& = (P_iU_i^*P_j^{\perp}U_j^*+P_i^{\perp}U_i^*P_jU_j^*)D_T+P_iU_i^*P_jU_j^*D_TT \\
				& \hspace{67mm} [\text{by } (\ref{identity 2})\text{ and condition-}(3) ]  \\
				& = P_iU_i^*(P_j^{\perp}U_j^*D_T+P_jU_j^*D_TT)+P_i^{\perp}U_i^*P_jU_j^*D_T\\
				& = P_iU_i^*D_TT_j+P_i^{\perp}U_i^*P_jU_j^*D_T. \hspace{27mm} [\text{by condition-(1)}]
			\end{align*}
		  
Hence it follows that  $V_jV_i=V_iV_j$ for all $ i,j $ and consequently $(V_1,\ldots ,V_n)$ is a commuting tuple.\\
			
			\noindent \textbf{\textit{Step 2.}} We now prove that each $V_j$ is an isometry and that $(V_1, \dots , V_n)$ is an isometric dilation of $(T_1, \dots , T_n)$. Note that		     
	   \[
	   	V_j^*V_j=	\begin{bmatrix}
	   		T_j^*T_j+D_TU_jP_jU_j^*D_T&0&0&\ldots\\
	   		U_jP_j^{\perp}P_jU_j^*D_T&U_jP_j^{\perp}U_j^*+U_jP_jU_j^*&U_jP_jP_j^{\perp}U_j &\ldots \\
	   		0&U_jP_j^{\perp}P_jU_j^*&U_jP_j^{\perp}U_j^*+U_jP_jU_j^*&\ldots  \\
	   		0&0&U_jP_j^{\perp}P_jU_j^*& \ldots   \\
	   		\ldots&\ldots&\ldots &\ddots
	   	\end{bmatrix}.
	   \]
By condition-(4), we have $T_j^*T_j+D_TU_jP_jU_j^*D_T=I$. Also, the identities $U_jP_j^{\perp}U_j^*+U_jP_jU_j^*=I$ and $U_jP_jP_j^{\perp}U_j^*=U_jP_j^{\perp}P_jU_j^*=0$ follow trivially. Thus, we have that $V_j^*V_j=I$ and hence $V_j$ is an isometry for $1\leq j \leq n$. It is evident from the block matrix of $V_i$ that $V_i^*|_{\mathcal{H}}=T_i^*$ and thus $(V_1,\ldots ,V_n)$ on $\mathcal{K}_0$ is an isometric dilation of $(T_1,\ldots ,T_n)$.\\

\noindent \textbf{\textit{Step 3.}} It remains to show that $\prod_{i=1}^nV_i=V$, where $V$ on $\mathcal{K}_0$ is the Sch$\ddot{a}$ffer's minimal isometric dilation of $T$. Note that $V$ has the block-matrix $V=\begin{bmatrix}
	T&0\\
	C&S
\end{bmatrix}$ with respect to the decomposition $\mathcal K_0=\mathcal{H}\oplus l^2(\mathcal{D}_T)$, where 
\[
 C=\begin{bmatrix}
	D_T\\0\\0\\ \vdots
\end{bmatrix}:\mathcal{H}\to l^2(\mathcal{D}_T) \text{ and } S=\begin{bmatrix}
	0&0&0&\cdots\\
	I&0&0&\cdots\\
	0&I&0&\cdots\\
	\vdots&\vdots&\vdots&\ddots
\end{bmatrix}:l^2(\mathcal{D}_T)\to l^2(\mathcal{D}_T).
\]	
Similarly for all $1\leq i \leq n$, $V_i$ has the following matrix form with respect to the decomposition $\mathcal{K}_0=\mathcal{H}\oplus l^2(\mathcal{D}_T)$:
\[ V_i=\begin{bmatrix}
	T_i&0\\
	\widetilde{C}_i & \widetilde{S}_i
\end{bmatrix},
\]
where
\[
\widetilde{C}_i=\begin{bmatrix}
	P_iU_i^*D_T\\0\\0\\ \vdots
\end{bmatrix}:\mathcal{H}\to l^2(\mathcal{D}_T) \text{ and } \widetilde{S}_i=\begin{bmatrix}
	P_i^{\perp}U_i^*&0&0&\cdots\\
	P_iU_i^*&P_i^{\perp}U_i^*&0&\cdots\\
	0&P_iU_i^*&P_i^{\perp}U_i^*&\cdots\\
	\vdots&\vdots&\vdots&\ddots
\end{bmatrix}:l^2(\mathcal{D}_T)\to l^2(\mathcal{D}_T).
\]
Up to a unitary $S\equiv M_z$ and $\widetilde{S}_i \equiv M_{P_i^{\perp}U_i^*+P_iU_i^*z}$ for $1\leq i \leq n$ on $H^2(\mathcal{D}_T)$. Further \eqref{identity 1} gives $P_i+U_i^*P_jU_j=P_j+U_j^*P_iU_j$. The conditions $(1)-(4)$ of Theorem \ref{BDF lemma} follow from the hypotheses of this theorem. Therefore, we have $\prod_{i=1}^n M_{P_i^{\perp}U_i^*+zP_iU_i^*}=M_z$ and consequently $S=\prod_{i=1}^n \widetilde{S}_i $. As observed by Bercovici, Douglas and Foias in \cite{Berc:Dou:Foi}, the terms involved in condition-(5) are all mutually orthogonal projections. This is because sum of projections $Q_1,Q_2$ is  again a projection if and only if they are mutually orthogonal. Now, suppose
\[
\underline{T_k}=T_1\ldots T_k ,\;\underline{U_k}=U_1\ldots U_k,\; \underline{P_k}=P_1+\underline{U_1}^*P_2\underline{U_1}+\ldots +\underline{U_{k-1}}^*P_k\underline{U_{k-1}}.
\]
Then clearly each $\underline{U_k}$ is a unitary and $\underline{P_k}$ is a projection. Let us define
	\[
	V_{\underline{T_k},\underline{U_k},\underline{P_k}}=\begin{bmatrix}
		\underline{T_k}&0&0&0&\ldots \\
		\underline{P_k}\;\underline{U_k}^*D_T&\underline{P_k}^{\perp}\underline{U_k}^*&0&0&\ldots \\
		0&\underline{P_k}\;\underline{U_k}^*&\underline{P_k}^{\perp}\underline{U_k}^*&0&\ldots \\
		0&0&\underline{P_k}\;\underline{U_k}^*&\underline{P_k}^{\perp}\underline{U_k}^*&\ldots \\
		\ldots & \ldots & \ldots & \ldots & \ldots
	\end{bmatrix}, \qquad 1\leq k \leq n. 
	\]    
	We prove that $V_{\underline{T_k},\underline{U_k},\underline{P_k}}V_{T_{k+1},U_{k+1},P_{k+1}}=V_{\underline{T_{k+1}},\underline{U_{k+1}},\underline{P_{k+1}}}$, 
	for all $1\leq k \leq n-1$. Note that for all $1\leq k \leq n$, $V_{\underline{T_k},\underline{U_k},\underline{P_k}}$ has the following block-matrix form with respect to the decomposition $\mathcal{K}_0=\mathcal{H}\oplus l^2(\mathcal{D}_T)$:
	\[ V_{\underline{T_k},\underline{U_k},\underline{P_k}}=\begin{bmatrix}
		\underline{T_k}&0\\
		\underline{C_k}&\underline{S_k}
	\end{bmatrix}, \] where \[ \underline{C_k}=\begin{bmatrix}
		\underline{P_k}\;\underline{U_k}^*D_T\\0\\0\\ \vdots
	\end{bmatrix}:\mathcal{H}\to l^2(\mathcal{D}_T) \text{ and } \underline{S_k}=\begin{bmatrix}
		\underline{P_k}^{\perp}\underline{U_k}^*&0&0&\cdots\\
		\underline{P_k}\;\underline{U_k}^*&\underline{P_k}^{\perp}\underline{U_k}^*&0&\cdots\\
		0&\underline{P_k}\;\underline{U_k}^*&\underline{P_k}^{\perp}\underline{U_k}^*&\cdots\\
		\vdots&\vdots&\vdots&\ddots
	\end{bmatrix}:l^2(\mathcal{D}_T)\to l^2(\mathcal{D}_T).
	\]
	It is clear from the definition that $\underline{U_{k+1}}=\underline{U_k}U_{k+1}$ and  $\underline{P_{k+1}}=\underline{P_k}+\underline{U_k}^*P_{k+1}\underline{U_k}$. Hence Lemma \ref{BDF lemma 1} tells us that $\underline{S_{k+1}}=\underline{S_k}S_{k+1}$. From the construction, it is clear that $\underline{T_{k+1}}=\underline{T_k}T_{k+1}$. Hence for proving $V_{\underline{T_k},\underline{U_k},\underline{P_k}}V_{T_{k+1},U_{k+1},P_{k+1}}=V_{\underline{T_{k+1}},\underline{U_{k+1}},\underline{P_{k+1}}}$, it suffices to prove $\underline{C_k}T_{k+1}+\underline{S_k}C_{k+1}=\underline{C_{k+1}}$. Here \[\underline{C_k}T_{k+1}+\underline{S_k}C_{k+1}= \begin{bmatrix}
		\underline{P_k}\;\underline{U_k}^*D_{T}T_{k+1}+\underline{P_k}^{\perp}\underline{U_k}^*P_{k+1}U_{k+1}^*D_{T}\\\underline{P_k}\;\underline{U_k}^*P_{k+1}U_{k+1}^*D_{T}\\0\\ \vdots
	\end{bmatrix}. \]
	Note that $\underline{S_{k+1}}=\underline{S_k}S_{k+1}$ gives
	\[
	\underline{P_k}\;\underline{U_k}^*P_{k+1}U_{k+1}^*=0 \quad \& \quad \underline{P_k}^{\perp}\underline{U_k}^*P_{k+1}U_{k+1}^*+\underline{P_k}\;\underline{U_k}^*P_{k+1}^{\perp}U_{k+1}^*=\underline{P_{k+1}}\;\underline{U^*_{k+1}}.
	\]
	Using these two equations and condition-$(1)$ of this theorem we obtain the following:
	\begin{align}
		\underline{P_k}\;\underline{U_k}^*D_{T}T_{k+1}+\underline{P_k}^{\perp}\underline{U_k}^*P_{k+1}U_{k+1}^*D_{T}&=\underline{P_k}\;\underline{U_k}^*D_{T}T_{k+1}+\underline{P_{k+1}}\;\underline{U_{k+1}}^*D_T-\underline{P_k}\;\underline{U_k}^*P_{k+1}^{\perp}U_{k+1}^*D_T   \notag  \\
		&=\underline{P_k}\;\underline{U_k}^*(D_{T}T_{k+1}-P_{k+1}^{\perp}U_{k+1}^*D_T)+\underline{P_{k+1}}\;\underline{U_{k+1}}^*D_T   \notag  \\
		&=\underline{P_k}\;\underline{U_k}^*P_{k+1}U_{k+1}^*D_TT+\underline{P_{k+1}}\;\underline{U_{k+1}}^*D_T \notag \\
		&=\underline{P_{k+1}}\;\underline{U_{k+1}}^*D_T.  \label{eqn:new-01}
	\end{align}	
This proves that $\underline{C_k}T_{k+1}+\underline{S_k}C_{k+1}=\underline{C_{k+1}}$ and hence $V_{\underline{T_k},\underline{U_k},\underline{P_k}}V_{T_{k+1},U_{k+1},P_{k+1}}=V_{\underline{T_{k+1}},\underline{U_{k+1}},\underline{P_{k+1}}}$ for all $1\leq k \leq n-1$. Therefore, by induction we have that $V_{\underline{T_n},\underline{U_n},\underline{P_n}}=V_1\ldots V_n $. Note that we have $ \underline{T_n}=T$, $\underline{U_n}=U_1\ldots U_n=I $. Also, it follows from condition-$(5)$ that $\underline{P_n}=I $. Therefore, $V_{\underline{T_n},\underline{U_n},\underline{P_n}}=V$, where $V$ is the Sch$\ddot{a}$ffer's minimal isometric dilation. Hence $\prod_{i=1}^n V_i=V$.\\

\noindent \textbf{\textit{Step-4.}} We now show that such an isometric dilation $(V_1, \dots , V_n)$ is minimal. Note that $V= \prod_{i=1}^n V_i$ is a minimal isometric dilation of $T= \prod_{i=1}^n T_i$. Therefore,
        \[
       \mathcal K = \overline{Span}\, \{ V^kh\,:\, h \in \mathcal H ,\, k \in \mathbb N \cup \{0 \}\, \}=  \overline{Span}\, \{ V_1^k\dots V_n^kh\,:\, h \in \mathcal H ,\, k \in \mathbb N \cup \{0 \}\, \}.
        \]
        Again
        \[
        Span\, \{ V_1^{k_1}\dots V_n^{k_n}h\,:\, h \in \mathcal H ,\, k_1,\dots , k_n \in \mathbb N \cup \{0 \}\, \} \subseteq \mathcal K.
        \]
        Therefore, 
        \begin{align*}
       \mathcal K & = \overline{Span}\, \{ V_1^k\dots V_n^kh\,:\, h \in \mathcal H ,\, k \in \mathbb N \cup \{0 \}\, \} \\
       & = \overline{Span}\, \{ V_1^{k_1}\dots V_n^{k_n}h\,:\, h \in \mathcal H ,\, k_1,\dots , k_n \in \mathbb N \cup \{0 \}\, \},
        \end{align*}
        and consequently $(V_1, \dots , V_n)$ is a minimal isometric dilation of $(T_1, \dots , T_n)$.\\

\noindent \textbf{(The $\Rightarrow$ part).} Let $(W_1,\ldots ,W_n)$ on $\mathcal{K}$ be an isometric dilation of $(T_1,\ldots ,T_n)$ such that $W=\Pi_{i=1}^nW_i$ is the minimal isometric dilation of $T$. Then $\mathcal K=\HS \oplus \HS^{\perp}$. Suppose $W_i'=\prod_{j \neq i}W_j$ for $1 \leq i \leq n$. So,
\[
\mathcal K= \overline{span} \, \{W^nh\,:\, h\in \HS, \; n \in \mathbb N \cup \{ 0\} \}.
\]
Since $V$ on $\mathcal K_0=\HS \oplus l^2(\mathcal D_T)$ is the minimal Sch$\ddot{a}$ffer's isometric dilation of $T$, it follows that
\[
\mathcal K_0 = \ov{span} \{ V^nh\,:\, h\in \HS, \; n \in \mathbb N \cup \{ 0\}  \}.
\]
Therefore, the map $\tau: \mathcal K_0 \rightarrow \mathcal K$ defined by $\tau(V^nh)=W^nh$ is a unitary which is identity on $\HS$. Thus, $\HS$ is a reducing subspace for $\tau$ and consequently
$\tau =\begin{pmatrix}
I & 0 \\
0 & \tau_1
\end{pmatrix}
$
for some unitary $\tau_1$. Then $W=\tau V \tau^* = \begin{bmatrix}
T & 0 \\
\tau_1C & \tau_1 S\tau_1^*
\end{bmatrix},
$		 
where $V=\begin{bmatrix}
T & 0 \\
C & S
\end{bmatrix}
$
with
\[
C=\begin{bmatrix}
		D_T\\
			0\\
			0\\
			\vdots
		\end{bmatrix}:\mathcal{H}\to l^2(\mathcal{D}_T) \quad \text{ and } \quad S= \begin{bmatrix}
		0&0&0&\cdots \\
		I&0&0&\cdots\\
		0&I&0&\cdots\\
		\vdots&\vdots&\vdots&\ddots
	\end{bmatrix}: l^2(\mathcal{D}_T) \to l^2(\mathcal{D}_T).
	\]
Let us consider the	commuting tuple of isometries $(\widehat{V}_1, \dots , \widehat{V}_n)=(\tau^* W_1 \tau, \dots , \tau^* W_n \tau)$. Note that $\prod_{i=1}^n \widehat{V}_i= \tau^* W \tau = V$. Define $\widehat{V}_i'=\prod_{j\neq i} \widehat{V}_j$ for $1 \leq i \leq n$.
Evidently each $\widehat{V}_i'$	is an isometry and $\widehat{V}_i, \widehat{V}_j', V$ commute for all $i,j$. Also, $\widehat{V}_i=\widehat{V}_i'^*V$ for $i=1, \dots , n$. Suppose the block-matrix of $\widehat{V}_i$ with respect to the decomposition $\HS \oplus l^2(\mathcal D_P)$ be $\widehat{V}_i=\begin{bmatrix}
T_i&A_i\\
C_i&S_i
\end{bmatrix}$. Now by the commutativity of $\widehat{V}_i$ and $\widehat{V}$, we have 
\begin{align}
\begin{bmatrix}
	T_i&A_i\\
	C_i&S_i
\end{bmatrix}
\begin{bmatrix}
T&0\\
C&S
\end{bmatrix}&=\begin{bmatrix}
T&0\\
C&S
\end{bmatrix}
\begin{bmatrix}
T_i&A_i\\
C_i&S_i
\end{bmatrix} \notag \\
i.e. \; \; 
\begin{bmatrix}
	T_iT+A_iC&A_iS\\
	C_iT+S_iC&S_iS
\end{bmatrix}&=\begin{bmatrix}
TT_i&TA_i\\
CT_i+SC_i&CA_i+SS_i
\end{bmatrix}.  \label{eqn:c01}
\end{align}
Since $T_i$ and $T$ commute, considering the $(1,1)$ position we have $A_iC=0$. We now show that $A_i=0$. Suppose $A_i=(A_{i1},A_{i2}\ldots )$ on $\mathcal{D}_T\oplus \mathcal{D}_T\oplus \ldots$ . Then the fact that $A_iC=0$ implies $A_{i1}=0$ on $\mathcal D_T$. Again $A_iS=TA_i$ gives
\[
\begin{bmatrix}
	0&A_{i2}&A_{i3}&\cdots
\end{bmatrix} \begin{bmatrix}
0&0&\ldots \\
I&0&\ldots\\
0&I&\ldots\\
\vdots&\vdots&\ldots
\end{bmatrix}=\begin{bmatrix}
0&TA_{i2}&TA_{i3}&\cdots
\end{bmatrix}
\] which further implies that $A_{i2}=0$ and $A_{ik}=TA_{i(k-1)}$ for all $k\geq 3$. Hence, inductively we have $A_{ik}=0$. This proves that $A_i=0$. A similar argument holds if we consider the commutativity of $\widehat{V}_i'$ and $V$. Thus, with respect to the decomposition $\mathcal{K}=\mathcal{H}\oplus l^2(\mathcal{D}_T)$, $\widehat{V}_i$ and $\widehat{V}_i'$ have the following block-matrix forms 
		\begin{equation} \label{eqn:new-0311}
		\widehat{V}_i=
		\begin{bmatrix}
			T_i&0\\
			C_i&S_i
		\end{bmatrix} \,,\qquad
		\widehat{V}_i'=
		\begin{bmatrix}
		T_i'&0\\
		C_i'&S_i'
	\end{bmatrix} \qquad \text{with} \quad T_i'={ \prod_{i\neq j} T_j}   \qquad (1 \leq i \leq n),
		\end{equation}
	for some bounded operators $C_i, C_i'$ and $S_i, S_i'$. It follows from the commutativity of $\widehat{V}_i, \widehat{V}_i'$ with $V$ that $S_i,S_i'$ commute with $S$. Evidently $S=M_z$ on $H^2(\mathcal{D}_T)$ and by being commutants of $S$, $S_i=M_{\phi}$ and $S_i'=M_{\phi_i'}$ for some $\phi_i, \phi_i' \in H^{\infty}(\mathcal B(\mathcal D_T))$. The relation $\widehat{V}_i=\widehat{V}_i'^*V$ gives
	\begin{equation} \label{eqn:0e1}
		\begin{bmatrix}
		T_i&0\\
		C_i&S_i
	\end{bmatrix}=\begin{bmatrix}
	T_i'^*&C_i'^*\\
	0&S_i'^*
\end{bmatrix} \begin{bmatrix}
T&0\\
C&S
\end{bmatrix}= \begin{bmatrix}
T_i'^*T+C_i'^*C&C_i'^*S\\
S_i'^*C&S_i'^*S
\end{bmatrix} , 
\end{equation}
where $T_i'$ is as in (\ref{eqn:new-0311}). Form here we have the following identities for each $i=1, \dots , n$:
	\begin{enumerate}
	\item[$(a)$] $T_i-T_i'^*T=C_i'^*C$, 
	\item[$(b)$] $C_i=(M_{\phi_i'})^*C$,
	\item[$(c)$] $M_{\phi_i}=(M_{\phi_i'})^*M_z. $
\end{enumerate}
Again, $\widehat{V}_i'=\widehat{V}_i^*V$ leads to
	\begin{equation} \label{eqn:0e2}
		\begin{bmatrix}
		T_i'&0\\
		C_i'&S_i'
	\end{bmatrix}=\begin{bmatrix}
	T_i^*&C_i^*\\
	0&S_i^*
\end{bmatrix} \begin{bmatrix}
T&0\\
C&S
\end{bmatrix}= \begin{bmatrix}
T_i^*T+C_i^*C&C_i^*S\\
S_i^*C&S_i^*S
\end{bmatrix}  
\end{equation}
	and hence we have for each $i=1, \dots , n$:
	\begin{enumerate}
	\item[$(a')$] $T_i'-T_i^*T=C_i^*C$, 
	\item[$(b')$] $C_i'=(M_{\phi_i})^*C$,
	\item[$(c')$] $M_{\phi_i'}=(M_{\phi_i})^*M_z. $
\end{enumerate}
From $(c)$ above and considering the power series expansions of $\phi_i$ and $\phi_i'$, we have that $\phi_i(z)=F_i+F_i'^*z$ and $\phi_i'(z)=F_i'+F_i^*z$ , for some $F_i, F_i' \in \mathcal B(\mathcal D_T)$. Therefore,
\[
S_i=M_{\phi_i}=\begin{bmatrix}
F_i&0&0&\cdots \\
F_i'^*&F_i&0&\cdots\\
0&F_i'^*&F_i&\cdots\\
\vdots&\vdots&\vdots&\ddots 	
\end{bmatrix} \;\; \text{ and } \;\;
S_i'=M_{\phi_i'}=\begin{bmatrix}
F_i'&0&0&\cdots \\
F_i^*&F_i'&0&\cdots\\
0&F_i^*&F_i'&\cdots\\
\vdots&\vdots&\vdots&\ddots 	
\end{bmatrix}.
\]
From $(b)$ and $(b')$ we have that
\[C_i=S_i'^*C=\begin{bmatrix}
	F_i'^*D_T\\
	0\\
	0\\
	\vdots 
\end{bmatrix} \quad \& \quad C_i'= S_i^*C=\begin{bmatrix}
	F_i^*D_T\\
	0\\
	0\\
	\vdots 
\end{bmatrix}.
\]
The fact that $\widehat{V}_i$ and $\widehat{V}_i'$ are isometries give us $F_i'F_i=F_iF_i'=0$, $F_i^*F_i+F_i'F_i'^*=F_i'^*F_i'+F_iF_i^*=I$ and $D_TF_i'F_i'^*D_T=D_{T_i}^2$ for $1\leq i \leq n$. Again, by the commutativity of  $\widehat{V}_i, \widehat{V}_j$ and considering the $(2,2)$ entries of their $2 \times 2$ block-matrices, we have the commutativity of $\phi_i$ and $\phi_j$. From here we have $[F_i,F_j]=0$ and $[F_i^*,F_j']=[F_j^*, F_i']$ and they hold for all $1 \leq i,j \leq n$. Similarly, by the commutativity of $\widehat{V}_i',\widehat{V}_j' $ we have that $[F_i',F_j']=0$ for each $i,j$. Thus, combining all together we have the following identities for $1\leq i,j \leq n$:
\begin{align}
&(i) \;\; F_iF_j  = F_jF_i \;\; \& \;\; [F_i^*,F_j']=[F_j^*, F_i'] \label{eqn:c02} \\
&(ii) \; \; F_i'F_j'  =F_j'F_i' \label{eqn:c03}\\
&(iii)\;\; F_i'F_i =F_iF_i'=0 \label{eqn:c04}\\
&(iv) \;\; F_i^*F_i+F_i'F_i'^* =F_i'^*F_i'+F_iF_i^*=I \label{eqn:c05} \\
& (v)\;\; D_TF_i'F_i'^*D_T=D_{T_i}^2. \label{eqn:c06}
\end{align}
For $1\leq  i \leq n$, let us define
$
U_i=F_i^*+F_i',\;\; U_{i}'=F_i^*-F_i' \;\; \text{ and } \; \; P_i=\frac{1}{2}(I-U_{i}'^*U_{i}).
$ 
Applying the above identities involving $F_i$ and $F_i'$ we have that $U_i, U_i'$ are unitaries. Note that $F_i=(U_i^*+U_i'^*)/2$ and $F_i'=(U_i-U_i')/2$. Hence $F_i'F_i=0$ implies that $U_iU_i'^*=U_i'U_i^*$ and $F_iF_i'=0$ implies that $U_i^*U_i'=U_i'^*U_i$. Thus $P_i=\dfrac{1}{2}(I-U_i^*U_i')$. It follows from here that $P_i$ is a projection. It can be verified that $F_i'=U_iP_i, F_i=P_i^{\perp}U_i^*.$ 
From (\ref{eqn:c01}), we have $C_iT+S_iC=CT_i+SC_i$ for each $i$ and substituting the values of $C_i, S_i$ and $C$ we have
\begin{equation} \label{eqn:c07}
F_i'^*D_TT+F_iD_T= D_TT_i
\end{equation}
We now show that the unitaries $U_1, \dots , U_n$ commute. For each $i,j$ we have
\begin{align*}
U_iU_j= & (F_i^*+F_i')(F_j^*+F_j')\\
=&F_i^*F_j^*+(F_i^*F_j'+F_i'F_j^*)+F_i'F_j'\\
=&F_j^*F_i^*+(F_j^*F_i'+F_j'F_i^*)+F_j'F_i' \quad [\text{ by the second part of } (\ref{eqn:c02})]\\
=&U_jU_i.
\end{align*}
Thus, substituting $F_i'=U_iP_i$, $F_i=P_i^{\perp}U_i^*$ we have from (\ref{eqn:c07}), (\ref{eqn:c02})-part-1, (\ref{eqn:c03}), (\ref{eqn:c06})
\begin{enumerate} 
		\item $D_TT_i=P_i^{\perp}U_i^*D_T+P_iU_i^*D_TT$,
	\item $P_i^{\perp}U_i^*P_j^{\perp}U_j^*=P_j^{\perp}U_j^*P_i^{\perp}U_i^*$,
\item $U_iP_iU_jP_j=U_jP_jU_iP_i$,
\item $D_TU_iP_iU_i^*D_T=D_{T_i}^2$,
   \end{enumerate}
respectively, where $U_1, \dots , U_n \in \mathcal B(\mathcal D_T)$ are commuting unitaries and $P_1, \dots , P_n \in \mathcal B(\mathcal D_T)$ are orthogonal projections. Since $\prod_{i=1}^n\widehat{V}_i=V$, the $(2,2)$ entry of the block-matrix gives $\prod_{i=1}^nS_i=S$. Note that $S_i=M_{\phi_i}$ for each $i$, where $\phi_i=F_i+zF_i'^* =P_i^{\perp}U_i^*+zP_iU_i^*$. So, we have $\prod_{i=1}^nM_{P_i^{\perp}U_i^*+zP_iU_i^*}=M_z$. Thus, by Lemma \ref{BDF lemma}, $\prod_{i=1}^n U_i=I$ and $(5)$ holds. \\

\noindent \textbf{\textit{Uniqueness}.} From $(a)$ and $(a')$ we have $
D_{T_i'}^2T_i=D_TF_i'D_T$ and $ D_{T_i}^2T_i'=D_TF_iD_T
$ respectively. If there is $G_i' \in \mathcal B ({\mathcal D_T})$ such that $D_{T_i'}^2T_i=D_TG_i'D_T$, then $D_T(F_i'-G_i')D_T=0$ and thus for any $h,g \in \HS$ we have
\[
\langle (F_i'-G_i')D_Th, D_Tg \rangle= \langle D_T(F_i'-G_i')D_Th,g \rangle =0. 
\]
This shows that $F_i'=G_i'$ and hence $F_i'$ is unique. Similarly one can show that $F_i$ is unique. Now $F_i=P_i^{\perp}U_i$ and $F_i'=U_iP_i$ and thus $U_i=F_i^*+F_i'$ and $P_i=F_i'^*F_i'$. Evidently the uniqueness of $F_i, F_i'$ gives the uniqueness of $U_i, P_i$ for $1\leq i \leq n$.\\

\noindent \textbf{(b).} The minimality of the dilation follows from the fact that $V=\prod_{i=1}^n V_i$ is the minimal isometric dilation of the product $T=\prod_{i=1}^n T_i$. Following the proof of the $(\Rightarrow)$ part of $\textbf{(a)}$ we see that any commuting isometric dilation $(W_1, \dots , W_n)$ of $(T_1, \dots , T_n)$ on a minimal isometric dilation space $\mathcal K_1$ of $T$ is unitarily equivalent to the isometric dilation $(V_1, \dots , V_n)$ on the Sch$\ddot{a}$ffer's minimal space $\mathcal K_0$. The rest of the argument follows and the proof is complete.

\end{proof}

\begin{rem} \label{comm:lift1}
Note that Theorem \ref{main} actually provides a commutant lifting in several variables. In Theorem \ref{main}, the isometric dilation $(V_1, \dots , V_n)$ of a commuting contractive tuple $(T_1, \dots , T_n)$ is constructed in such a way that the product $V=\prod_{i=1}^n V_i$ becomes the minimal isometric dilation of the contraction $T=\prod_{i=1}^n T_i$. Also, it is evident from the block-matrix form of $V_i$ as in (\ref{eqn:isom-01}) that $\HS$ is coinvariant under each $V_i$ and $V_i^*|_{\HS}=T_i^*$. Thus, each $T_i$ is a commutant of $T$ and is being lifted to $V_i$ which is a commutant of the minimal isometric dilation $V$ of $T$.

\end{rem}

\begin{note}
	
Ando's theorem tells us that every pair of commuting contractions $(T_1,T_2)$ dilates to a pair of commuting isometries $(V_1,V_2)$, but $(T_1,T_2)$ may not have such an isometric dilation $(V_1,V_2)$ such that $V=V_1V_2$ is the minimal isometric dilation of $T=T_1T_2$. The following example shows that there are commuting contractions $T_1,T_2$ that violate the conditions of Theorem \ref{main}.
    
\end{note}

\begin{eg} \label{exmp:06}
	Let us consider the following contractions on $\mathbb C^3$ :
	\[
	T_1=\begin{pmatrix}
		0&0&0\\
		1/3&0&0\\
		0&1/3\sqrt{3}&0
	\end{pmatrix} \quad \& \quad
	T_2=\begin{pmatrix}
		0&0&0\\
		0&0&0\\
		-1/\sqrt{3}&0&0
	\end{pmatrix}.
	\]
	Evidently $T=T_1T_2=0$ and thus $D_T=I$. We show that $(T_1, T_2)$ does not dilate to a commuting pair of isometries acting on the minimal dilation space of $T$. Suppose it happens. Then there exist projections $P_1,P_2$ and commuting unitaries $U_1,U_2$ in $\mathcal{B}(\mathcal{D}_T)$ with $U_1U_2=I$ satisfying conditions $(1)-(5)$ of Theorem \ref{main}. Following the arguments in the $(\Rightarrow)$ part of the proof of Theorem \ref{main}, we see that $T_1,T_2$ satisfy $(a)$ for $i=1,2$ (see the first display after (\ref{eqn:0e1})), i.e.
	\[
	D_TU_1P_1D_T=T_2 ,\;\; \& \;\; D_TU_2P_2D_T=T_1.
	\]
	 Since $D_T=I$, we have that $U_1P_1=T_2$ and $U_2P_2=T_1$. Now we have
	\[
	D_TU_1P_1U_1^*D_T=U_1P_1U_1^*=T_2T_2^*
	=\begin{pmatrix}
		0&0&0\\0&0&0\\-1/\sqrt{3}&0&0
	\end{pmatrix}\begin{pmatrix}
		0&0&-1/\sqrt{3}\\
		0&0&0\\
		0&0&0
	\end{pmatrix}=\begin{pmatrix}
		0&0&0\\
		0&0&0\\
		0&0&1/3
	\end{pmatrix}.
	\]
	Also,
	\[
	D_{T_1}^2=I-T_1^*T_1=\begin{pmatrix}
		1&0&0\\
		0&1&0\\
		0&0&1
	\end{pmatrix}-\begin{pmatrix}
		0&1/3&0\\
		0&0&1/3\sqrt{3}\\
		0&0&0
	\end{pmatrix}\begin{pmatrix}
		0&0&0\\
		1/3&0&0\\
		0&1/3\sqrt{3}&0
	\end{pmatrix}.
	\]
	Thus
	\[
	D_{T_1}^2=\begin{pmatrix}8/9&0&0\\0&26/27&0\\0&0&1 \end{pmatrix}.
	\]
	Hence $D_TU_1P_1U_1^*D_T\neq D_{T_1}^2$ which contradicts condition-$(4)$ of Theorem \ref{main}. Hence $(T_1,T_2)$ does not dilate to a pair of commuting isometries $(V_1,V_2)$ acting on the minimal isometric dilation space of $T=T_1T_2$.
\end{eg}

It is evident from the first part of the proof of Theorem \ref{main} that such an isometric dilation $(V_1, \dots , V_n)$ can be constructed for $(T_1,\dots, T_n)$ with the help of conditions $(1)-(4)$ of Theorem \ref{main} and without even assuming that $\prod_{i=1}^n U_i=I$. Condition-$(5)$ was to make the product $\prod_{i=1}^n V_i$ the minimal isometric dilation of $\prod_{i=1}^n T_i$. Thus, a different version of Theorem \ref{main} can be presented based on a weaker hypothesis in the following way. Needless to mention that a proof to this follows naturally from the proof of Theorem \ref{main}.
	
\begin{thm}\label{coromain}
	Let $T_1,\ldots, T_n \in \mathcal{B}(\mathcal{H})$ be commuting contractions and let $T=\prod_{i=1}^nT_i$. Then $(T_1,\ldots ,T_n)$ possesses an isometric dilation on the minimal isometric dilation space of $T$, if there are projections $P_1,\ldots ,P_n$ and commuting unitaries $U_1,\ldots ,U_n$ in $\mathcal{B}(\mathcal{D}_T)$ such that the following hold for $i=1, \dots, n$:
	\begin{enumerate} 
		\item $D_TT_i=P_i^{\perp}U_i^*D_T+P_iU_i^*D_TT$ ,
		\item  $P_i^{\perp}U_i^*P_j^{\perp}U_j^*=P_j^{\perp}U_j^*P_i^{\perp}U_i^*$ ,
		\item $U_iP_iU_jP_j=U_jP_jU_iP_i$ ,
		\item $D_TU_iP_iU_i^*D_T=D_{T_i}^2$.
	\end{enumerate} 
	Conversely, if $(T_1,\ldots ,T_n)$ possesses an isometric dilation $(\widehat{V}_1,\ldots,\widehat{V}_{n})$ with $V=\prod_{i=1}^nV_i$ being the minimal isometric dilation of $T$, then there are unique projections $P_1,\ldots ,P_n$ and unique commuting unitaries $U_1,\ldots ,U_n$ in $\mathcal{B}(\mathcal{D}_T)$ satisfying the conditions $(1)-(4)$ above.
	
\end{thm}

\vspace{0.1cm}

\section{Minimal isometric dilation and functional model when the product is a $C._0$ contraction}

\vspace{0.2cm}

\noindent In this Section, we consider a tuple of commuting contractions $(T_1, \dots , T_n)$ with the product $T=\prod_{i=1}^n T_i$ being a $C._0$ contraction, i.e. ${T^*}^n \rightarrow 0$ strongly as $n \rightarrow \infty$. Such a tuple dilates (on the minimal isometric dilation space of $T$) with a weaker hypothesis than that of Theorem \ref{main} as the following theorem shows. This is another main result of this article.

\begin{thm}\label{puredil}
	Let $T_1,\ldots ,T_n$ be commuting contractions on a Hilbert space $\mathcal{H}$ such that their product $T=\prod_{i=1}^nT_i$ is a $C._0$ contraction. Then $(T_1,\ldots ,T_n)$ possesses an isometric dilation $(V_1,\ldots ,V_n)$ with $V=\prod_{i=1}^nV_i$ being a minimal isometric dilation of $T$ if and only if there are unique orthogonal projections $P_1,\ldots ,P_n$ and unique commuting unitaries $U_1,\ldots ,U_n$ in $\mathcal{B}(\mathcal{D}_{T^*})$ with $\prod_{i=1}^n U_i=I$ such that the following conditions hold for $i=1, \dots , n$ :
	\begin{enumerate}
		\item $D_{T^*}T_i^*=P_i^{\perp}U_i^*D_{T^*}+P_iU_i^*D_{T^*}T^*$,
		\item $P_i^{\perp}U_i^*P_j^{\perp}U_j^*=P_j^{\perp}U_j^*P_i^{\perp}U_i^*$ ,
		\item $U_iP_iU_jP_j=U_jP_jU_iP_i$ ,
		\item  $P_1+U_1^*P_2U_1+U_1^*U_2^*P_3U_2U_1+\ldots +U_1^*U_2^*\ldots U_{n-1}^*P_nU_{n-1}\ldots U_2U_1 =I_{\mathcal{D}_{T^*}}$. 
	\end{enumerate} 
\end{thm}
\begin{proof}
	First we assume that there exist orthogonal projections $P_1,\ldots ,P_n$ and commuting unitaries $U_1,\ldots ,U_n$ in $\mathcal{B}(\mathcal{D}_{T^*})$ such that $\prod_{i=1}^nU_i=I$ and the above conditions $(1)-(4)$ hold. Since $T$ is a $C._0$ contraction, $H^2(\mathcal D_{T^*})$ is a minimal isometric dilation space for $T$. We first construct an isometric dilation $(V_1, \dots , V_n)$ for $(T_1, \dots , T_n)$ on the minimal isometric dilation space $H^2(\mathcal D_{T^*})$ of $T$ with the assumptions $(1)-(3)$ only. Condition-$(4)$ will imply that $\prod_{i=1}^n V_i=V$, the minimal isometric dilation of $T$. Let us consider the following multiplication operators acting on $H^2(\mathcal D_{T^*})$:
	\[
	V_i=M_{U_iP_i^{\perp}+zU_iP_i} \qquad (1\leq i \leq n).
	\]
	 Since $U_i$ is a unitary and $P_i$ is a projection, it follows that
	\[
	(P_i^{\perp}U_i^*+\bar{z}P_iU_i^*)(U_iP_i^{\perp}+zU_iP_i)=P_i^{\perp}+P_i=I
	\]
	and thus $V_i$ is an isometry for each $i$. Now we show that $(V_1, \dots , V_n)$ is a commuting tuple. Note that for any $1\leq i <j\leq n$, $V_iV_j=V_jV_i$ if and only if
	\[
	(U_iP_i^{\perp}+zU_iP_i)(U_jP_j^{\perp}+zU_jP_j)=(U_jP_j^{\perp}+zU_jP_j)(U_iP_i^{\perp}+zU_iP_i)
	\]
	which happens if and only if the given conditions $(2), (3)$ hold along with
	\begin{equation}\label{T2}
U_iP_iU_jP_j^{\perp}+U_iP_i^{\perp}U_jP_j=U_jP_j^{\perp}U_iP_i+U_jP_jU_iP_i^{\perp}
\end{equation}
which we prove now. From the given condition-$(2)$ we have $[U_iP_i^{\perp},U_jP_j^{\perp}]=0$ and this implies that
\[
U_iU_j-U_iU_jP_j-U_iP_iU_j+U_iP_iU_jP_j-U_jU_i+U_jU_iP_i+U_jP_jU_i-U_jP_jU_iP_i=0.
\]
From here we have that
\begin{equation} \label{T1}
	U_iP_iU_j+U_iU_jP_j=U_jU_iP_i+U_jP_jU_i 
\end{equation}
We now prove that
\begin{equation} \label{eqn:new-03}
(U_iP_iU_jP_j^{\perp}+U_iP_i^{\perp}U_jP_j)=(U_jP_j^{\perp}U_iP_i+U_jP_jU_iP_i^{\perp}).
\end{equation}
We have that
	\begin{align*}
&(U_iP_iU_jP_j^{\perp}+U_iP_i^{\perp}U_jP_j)-(U_jP_j^{\perp}U_iP_i+U_jP_jU_iP_i^{\perp})\\
&=U_iP_iU_jP_j^{\perp}-U_jP_j^{\perp}U_iP_i-U_jP_jU_iP_i^{\perp}+U_iP_i^{\perp}U_jP_j\\
&=  U_iP_iU_j-U_iP_iU_jP_j-U_jU_iP_i+U_jP_jU_iP_i-U_jP_jU_i+U_jP_jU_iP_i+U_iU_jP_j-U_iP_iU_jP_j\\
&=U_iP_iU_j-U_jU_iP_i-U_jP_jU_i+U_iU_jP_j \\
&=0.\; \qquad [\text{ by } \eqref{T1}]
\end{align*}
Therefore, (\ref{T2}) is proved and consequently $(V_1, \dots , V_n)$ is a tuple of commuting isometries. We now embed $\HS$ inside $H^2(\mathcal D_{T^*})$. Let us define $W:\mathcal{H}\to H^2(\mathcal{D}_{T^*})$ by
\begin{equation} \label{eqn:0f1}
 W(h)=\sum_{0}^{\infty}z^n D_{T^*}T^{*n}h .
 \end{equation}
 It can be found in the literature (e.g. \cite{Nagy}) that the map $W$ is an isometry. However, we include a proof here for the sake of completeness and for the convenience of a reader.
 \begin{align*}
	\|Wh\|^2 =\|\sum_{n=0}^{\infty}z^nD_{T^*}T^{*n}h\|^2
	&= \langle \sum_{n=0}^{\infty}z^nD_{T^*}T^{*n}h, \sum_{m=0}^{\infty}z^mD_{T^*}T^{*m}h \rangle\\
	&= \sum_{m,n=0}^{\infty}\langle z^n,z^m\rangle \langle  D_{T^*}T^{*n}h, D_{T^*}T^{*m}h \rangle\\
	&= \sum _{n=0}^{\infty}  \langle T^nD_{T^*}^2T^{*n}h,h \rangle\\
	&= \sum_{n=0}^{\infty} \langle T^n(I-TT^*)T^{*n}h,h \rangle \\
	&= \sum_{n=0}^{\infty} (\langle T^nT^{*n}h,h\rangle - \langle T^{n+1}T^{*(n+1)}h,h\rangle ) \\
	&= \lim_{m\to \infty }\sum_{n=0}^{m}(\|T^{*n}h\|^2-\|T^{*n+1}h\|^2)\\ 
	&= \|h\|^2 - \lim_{m\to \infty }\|T^{*m}h\|^2\\
	&=\|h\|^2. 
\end{align*} 
The second last equality follows from the fact that ${\displaystyle \lim_{n\to \infty }\|T^{*n}h\|^2=0}$ as $T$ is a $C._0$ contraction. Thus, $W$ is an isometry.
We now determine the adjoint of $W$. For any $n\geq 0$, $\xi \in \mathcal{D}_{T^*}$, we have
\[
	\langle W^*(z^n\xi),h \rangle = \langle z^n\xi, \sum_{n=0}^{\infty}z^nD_{T^*}T^{*n}h  \rangle = \langle \xi, D_{T^*}T^{*n}h \rangle = \langle T^nD_{T^{*}}\xi, h \rangle .
\]
Therefore, $W^*(z^n\xi)=T^nD_{T^*}\xi$.	Now for any $1\leq i \leq n$, for all $k\in \mathbb{N}\cup \{ 0 \}$ and for each $\xi \in \mathcal{D}_{T^*}$ we have 
	\begin{align*}
		W^*V_i(z^k\xi) = W^*M_{U_iP_i^{\perp}+zU_iP_i}(z^k\xi) = W^*(z^kU_iP_i^{\perp}\xi+z^{k+1}U_iP_i\xi)
		&=T^kD_{T^*}U_iP_i^{\perp}\xi+T^{k+1}D_{T^*}U_iP_i\xi \\
		&=T^k(D_{T^*}U_iP_i^{\perp}\xi+TD_{T^*}U_iP_i\xi)\\
		&=T^k(T_iD_{T^*}\xi) \;\; [\text{ by condition-}(1)]\\
		&=T_i(T^kD_{T^*}\xi)\\
		&=T_iW^*(z^k\xi).
	\end{align*}
	Therefore, $W^*V_i=T_iW^*$, i.e. $V_i^*W=WT_i^*=WT_i^*W^*W$ and hence $V_i^*|_{W(\mathcal{H})}=WT_i^*W^*|_{W(\mathcal{H})}$. This proves that $(V_1,\ldots ,V_n)$ is an isometric dilation of $(T_1,\ldots ,T_n)$. Now \eqref{T1} gives us $P_i+U_i^*P_jU_i=P_j+U_j^*P_iU_j$ for $1\leq i<j\leq n$ and condition-(4) yields $P_i+U_i^*P_jU_i=P_j+U_j^*P_iU_j\leq I_{\mathcal{D}_{T^*}}$. Hence, by an application of Lemma \ref{BDF lemma O} we have that that $ \prod_{i=1}^nV_i=M_z$, which is (upto a unitary) the minimal isometric dilation of $T$.
	
	\medskip
	
Conversely, suppose $(Y_1, \dots , Y_n)$ acting on $\mathcal K$ is an isometric dilation of $(T_1, \dots , T_n)$, where $Y={ \prod_{i=1}^nY_i}$ is the minimal isometric dilation of $T$. Let $Y_i'={ \prod_{j \neq i}Y_j}$ for $1 \leq i \leq n$. Then,
\[
\mathcal{K}=\overline{span}\{ Y^nh:h\in \mathcal{H}, \, n\in \mathbb{N}\cup \{0\}  \}.
\]
We first show that $Y_i^*|_{\HS}=T_i^*$ for each $i=1, \dots , n$. Note that for any $i=1,\ldots ,n$, $k\in \mathbb{N}\cup \{0 \}$ and $h\in \mathcal{H}$, $T_iP_{\mathcal{H}}(Y^kh) = T_iT^kh = P_{\mathcal{H}}Y_i(Y^kh) $ and as a consequence we have $T_iP_{\mathcal{H}}=P_{\mathcal{H}}Y_i$. Now, for any $h\in \mathcal{H}$ and $k\in \mathcal{K}$,
	\[
		\langle Y_i^*h, k \rangle=\langle h,Y_ik \rangle = \langle h, P_{\mathcal{H}}Y_ik\rangle = \langle h,T_iP_{\mathcal{H}}k \rangle = \langle T_i^*h,P_{\mathcal{H}}k \rangle = \langle T_i^*h,k \rangle .
	\]
Hence $Y_i|_{\HS}=T_i$ for each $i$. Therefore, the block matrix form of each $Y_i$ with respect to the decomposition $\mathcal K = \mathcal{H}\oplus {\HS}^{\perp}$ is $Y_i=\begin{bmatrix}
		T_i&0\\
		C_i&S_i
	\end{bmatrix}$ for some operators $C_i,\, S_i$. Also, since $Y=\Pi_{i=1}^n Y_i$, we have $Y=\begin{bmatrix}
	T&0\\
	C&S
\end{bmatrix}$ for some operators $C,\, S $. Again $V= { \prod_{i=1}^n V_i} = M_z$ on $H^2(\mathcal D_{T^*})$ is also a minimal isometric dilation of $T$. Therefore,
 \[
 H^2(\mathcal{D}_{T^*})=\overline{span}\{V^n(Wh):h\in \mathcal{H},\,n\in \mathbb{N}\cup \{ 0\}  \}.
 \]
Now the map $\tau :H^2(\mathcal{D}_{T^*})\to \mathcal{K}$ defined by $\tau(V^nWh)=Y^nh$ is a unitary which maps $Wh$ to $h$ and $\tau^*$ maps $h$ to $Wh$ for all $h\in\mathcal{H}$. Therefore, with respect to the decomposition $\mathcal{K}=\mathcal{H}\oplus \mathcal{H}^{\perp}$ and $H^2(\mathcal{D}_{T^*})=W(\mathcal{H})\oplus (W\mathcal{H})^{\perp}$ the map $\tau$ has block matrix form $\tau=\begin{bmatrix}
		W^*&0\\
		0&\tau_{1}
	\end{bmatrix}$ for some unitary $\tau_1$. Evidently $V=\tau^* Y \tau$. Now
	\[
	\tau^*Y_i\tau = \begin{bmatrix}
		W&0\\
		0&\tau_{1}^*
	\end{bmatrix} \begin{bmatrix}
	T_i&0\\
	C_i&S_i
\end{bmatrix}\begin{bmatrix}
W^*&0\\
0&\tau_{1}
\end{bmatrix} = \begin{bmatrix}
	WT_i&0\\
	\tau_{1}^*C_i&\tau_{1}^*S_i
\end{bmatrix}\begin{bmatrix}
W^*&0\\
0&\tau_{1}
\end{bmatrix} = \begin{bmatrix}
	WT_iW^*&0\\
	\tau_{1}^*C_iW^*&\tau_{1}^*S_i\tau_{1}
\end{bmatrix}.
\] 
Therefore, $\tau^*V_i^*\tau(Wh)=WT_i^*W^*(Wh)=WT_i^*h$.  For each $i=1, \dots , n$, let us define
$
\widehat{V}_i:=\tau^*V_i\tau $. Therefore, \begin{equation}\label{Vistar}
	\widehat{V}_i^*Wh=WT_i^*h\, \quad \text{ for all } \; h \in \HS,\qquad (1\leq i \leq n).
\end{equation}
Therefore, $(\widehat{V}_1, \dots , \widehat{V}_n)=(\tau^*Y_1 \tau, \dots , \tau^* Y_n \tau)$ on $H^2(\mathcal D_{T^*})$ is an isometric dilation of $(T_1, \dots , T_n)$ such that ${ \prod_{i=1}^n \widehat{V}_i}=V$. We now follow the arguments given in the converse part of the proof of Theorem \ref{main}. Since $\widehat{V}_i$ is a commutant of $V \;(=M_z)$, $\widehat{V}_i=M_{\phi_i}$, where $\phi_i(z)=F_i'^*+F_iz\in H^{\infty}(\mathcal B(\mathcal D_{T^*}))$. Evidently $U_i=F_i^*+F_i',\, U_{i}'=F_i^*-F_i'$ are commuting unitaries and $P_i=\dfrac{1}{2}(I-U_{i}'^*U_i)$ is a projection for all $i=1, \dots ,n$. A simple computation shows that $F_i=U_iP_i$ and $F_i'=P_i^{\perp}U_i^*$. Also, $[F_i, F_j]=[F_i', F_j']=0$ for all $i,j$. Therefore,
\[
\widehat{V}_i=M_{U_iP_i^{\perp}+U_iP_iz} \qquad (1 \leq i \leq n).
\]
Obviously condition-(2) and condition-(3) follow from the commutativity of $F_i, F_j$ and $F_i',F_j'$ respectively. It remains to show that conditions $(1)$ and $(4)$ hold. From \eqref{Vistar}, we have $\widehat{V}_i^*Wh=WT_i^*h$ for every $h \in \HS$. Now for any $h\in \mathcal{H}$, we have
\begin{align*}
	\widehat{V}_i^*Wh&=\widehat{V}_i^*(\sum_{k=0}^{\infty}z^kD_{T^*}T^{*k}h )\\
	&=P_i^{\perp}U_i^*D_{T^*}h+\sum_{k=1}^{\infty} z^kP_i^{\perp}U_i^*D_{T^*}T^{*k}h +z^{k-1}P_iU_i^*D_{T^*}T^{*k}h \\
	&=\sum_{k=0}^{\infty}z^k(P_i^{\perp}U_i^*D_{T^*}T^{*k}+P_iU_i^*D_{T^*}T^{*(k+1)})h\\
	&=\sum_{k=0}^{\infty}z^k(P_i^{\perp}U_i^*D_{T^*}+P_iU_i^*D_{T^*}T^*)T^{*k}h.
\end{align*}
Also,
\[
W(T_i^*h)=\sum_{k=0}^{\infty}z^kD_{T^*}T^{*k}T_i^*h	=\sum_{k=0}^{\infty}z^kD_{T^*}T_i^*T^{*k}h.
\]
Comparing the constant terms we have
   $
   D_{T^*}T_i^* =P_i^{\perp}U_i^*D_{T^*}+P_iU_i^*D_{T^*}T^*.
   $ This proves  condition-$(1)$. Now, since we have $ \prod_{i=1}^n\widehat{V}_i=V=M_z$ and $\widehat{V}_i=M_{U_iP_i^{\perp}+zU_iP_i}$, condition-$(4)$ follows from Theorem \ref{BDF lemma O}. The uniqueness of $P_1, \dots , P_n$ and $U_1, \dots , U_n$ follows by an argument similar to that in the proof of the $(\Rightarrow)$ part of Theorem \ref{main}. The proof is now complete.	

\end{proof}

Now we present an analogue of Theorem \ref{coromain} when the product $T$ is a $C._0$ contraction and obviously a proof follows from Theorem \ref{puredil} and its proof.

\begin{thm}\label{coropuredil}
	Let $T_1,\ldots, T_n \in \mathcal{B}(\mathcal{H})$ be commuting contractions such that $T=\prod_{i=1}^nT_i$ is a $C._0$ contraction. Then $(T_1,\ldots ,T_n)$ possesses an isometric dilation on the minimal isometric dilation space of $T$, if there are projections $P_1,\ldots ,P_n$ and commuting unitaries $U_1,\ldots ,U_n$ in $\mathcal{B}(\mathcal{D}_{T^*})$ such that the following hold for $i=1, \dots, n$:
	\begin{enumerate} 
		\item $D_TT_i=P_i^{\perp}U_i^*D_T+P_iU_i^*D_TT$ ,
		\item  $P_i^{\perp}U_i^*P_j^{\perp}U_j^*=P_j^{\perp}U_j^*P_i^{\perp}U_i^*$ ,
		\item $U_iP_iU_jP_j=U_jP_jU_iP_i$ .
	\end{enumerate} 
	Conversely, if a commuting tuple of contractions $(T_1,\ldots ,T_n)$, with the product $T=\prod_{i=1}^n T_i$ being a $C._0$ contraction, possesses an isometric dilation $(\widehat{V}_1,\ldots,\widehat{V}_{n})$, where $V=\prod_{i=1}^nV_i$ is the minimal isometric dilation of $T$, then there are unique projections $P_1,\ldots ,P_n$ and unique commuting unitaries $U_1,\ldots ,U_n$ in $\mathcal{B}(\mathcal{D}_{T^*})$ satisfying the conditions $(1)-(3)$ above. 
\end{thm}

\subsection{A functional model when the product is a $C._0$ contraction} For a contraction $T$ acting on a Hilbert space $\mathcal H$, let
$\Lambda_T$ be the set of all complex numbers for which the
operator $I-zT^*$ is invertible. For $z\in \Lambda_T$, the
characteristic function of $T$ is defined as
\begin{eqnarray}\label{e0}
\Theta_T(z)=[-T+zD_{T^*}(I-zT^*)^{-1}D_T]|_{\mathcal D_T}.
\end{eqnarray}
Here we recall a few definitions and terminologies from the initial part of Section \ref{sec:03}. The operators $D_T$ and $D_{T^*}$ are the
defect operators $(I-T^*T)^{1/2}$ and $(I-TT^*)^{1/2}$
respectively. By virtue of the relation $TD_T=D_{T^*}T$ (section
I.3 of \cite{Nagy}), $\Theta_T(z)$ maps $\mathcal
D_T=\overline{\textup{Ran}}D_T$ into $\mathcal
D_{T^*}=\overline{\textup{Ran}}D_{T^*}$ for every $z$ in
$\Lambda_T$.

\smallskip

In \cite{Nagy}, Sz.-Nagy and Foias proved that every
$C._0$ contraction $P$ acting on $\mathcal H$ is
unitarily equivalent to the operator $\mathbb T=P_{\mathbb
H_T}M_z|_{\mathbb
H_T}$ on the Hilbert space
$\mathbb H_T=H^2(\mathcal D_{T^*}) \ominus
M_{\Theta_T}(H^2(\mathcal D_T))$, where $M_z$ is
the multiplication operator on $H^2(\mathcal D_{T^*})$ and $M_{\Theta_T}$
is the multiplication operator from $H^2(
\mathcal D_T)$ into $H^2(\mathcal D_{T^*})$
corresponding to the multiplier $\Theta_T$. This is known
as Sz.Nagy-Foias model for a $C._0$ contraction. Indeed, $M_z$ on $H^2(\mathcal D_{T^*})$ dilates $T \in \mathcal B(\HS)$ and $W: \HS \rightarrow H^2(\mathcal D_{T^*})$ as in (\ref{eqn:0f1}) is the concerned isometric embedding. In an analogous manner by an application of Theorem \ref{puredil}, we obtain a functional model for a tuple of commuting contractions with $C._0$ product. A notable fact about this model is that the multiplication operators involved in this model have analytic symbols which are linear functions in one variable.

\begin{thm}\label{thm:pure-model}
Let $T_1,\ldots ,T_n$ be commuting contractions on a Hilbert space $\mathcal{H}$ such that their product $T=\Pi_{i=1}^nT_i$ is a $C._0$ contraction. If there are projections $P_1,\ldots ,P_n$ and commuting unitaries $U_1,\ldots ,U_n$ in $\mathcal{B}(\mathcal{D}_{T^*})$ satisfying the following for $1 \leq i \leq n$,
	\begin{enumerate}
		\item $D_{T^*}T_i^*=P_i^{\perp}U_i^*D_{T^*}+P_iU_i^*D_{T^*}T^*$,
		\item $P_i^{\perp}U_i^*P_j^{\perp}U_j^*=P_j^{\perp}U_j^*P_i^{\perp}U_i^*$ ,
		
\item $U_iP_iU_jP_j=U_jP_jU_iP_i$ ,
	\end{enumerate}
	Then $(T_1, \dots , T_n)$ is unitarily equivalent to $(\widetilde{T}_1, \dots , \widetilde{T}_n)$ acting on the space $\mathbb H_T=H^2(\mathcal D_{T^*})\ominus M_{\Theta_T}(H^2(\mathcal D_T))$, where $\widetilde{T}_i=P_{\mathbb H_T}(M_{U_iP_i^{\perp}+zU_iP_i})|_{\mathbb H_T}$ for $1\leq i \leq n$.
\end{thm}

\begin{proof}

For a $C._0$ contraction $T$, we have from literature (e.g. \cite{Nagy} or Lemma 3.3 in \cite{tirtha-sourav2}) that
\[
WW^*+M_{\Theta_T}M_{\Theta_T}^*=I_{H^2(\mathcal
 D_{P^*})}.
\]
It follows from here that $W(\mathcal H)=\mathbb H_T$, where $W: \HS \rightarrow H^2(\mathcal D_{T^*})$ is as in (\ref{eqn:0f1}). Since $V_i^*|_{W(\HS)}=T_i^*$, where $V_i=(M_{U_iP_i^{\perp}+zU_iP_i})$, we have that $T_i \cong P_{\mathbb H_T}(M_{U_iP_i^{\perp}+zU_iP_i})|_{\mathbb H_T}$ for $i=1, \dots , n$.
 
\end{proof}

\subsection{A factorization of a $C._0$ contraction.} The model for commuting $n$-isometries, Theorem \ref{BDF lemma O}, can be restated in the following way.
\begin{thm}
	Let $V_1,\dots ,V_n$ be commuting isometries acting on a Hilbert space $\mathcal{H}$ and let $V=\prod_{i=1}^nV_i$. Then $V=\prod_{i=1}^nV_i$ is a pure isometry if and only if there are unique orthogonal projections $P_1,\ldots ,P_n$ and unique commuting unitaries $U_1,\ldots ,U_n$ in $\mathcal{B}(\mathcal{D}_{V^*})$ with $\prod_{i=1}^n U_i=I$ such that the following conditions hold for $i=1, \dots , n$ :
	\begin{enumerate}
		\item $P_i^{\perp}U_i^*P_j^{\perp}U_j^*=P_j^{\perp}U_j^*P_i^{\perp}U_i^*$ ,
		\item $U_iP_iU_jP_j=U_jP_jU_iP_i$ ,
		\item  $P_1+U_1^*P_2U_1+U_1^*U_2^*P_3U_2U_1+\ldots +U_1^*U_2^*\ldots U_{n-1}^*P_nU_{n-1}\ldots U_2U_1 =I_{\mathcal{D}_T}$. 
	\end{enumerate}
	Moreover, $(V_1, \dots , V_n)$ is unitarily equivalent to $(M_{U_1P_1^{\perp}+zU_1P_1}, \dots , M_{U_nP_n^{\perp}+zU_nP_n})$ on $H^2(\mathcal D_{V^*})$. 
\end{thm}
\begin{proof}
	First suppose $(V_1,\dots, V_n )$ on Hilbert space $\mathcal{H}$ is a commuting $n$ tuple of isometries with $\prod_{i=1}^nV_i=V$ being a pure isometry. Thus, $\mathcal H$ can be identified with $H^2(\mathcal D_{V^*})$ via a unitary $\tau:\mathcal{H}\to H^2(\mathcal{D}_{V^*})$ and $V$ can be identified with $M_z$ on $H^2(\mathcal{D}_{V^*})$. Let $\widehat{V}_i=\tau V_i \tau^*$ for $i=1,2\dots, n $. Hence $(\widehat{V}_1, \dots ,\widehat{V}_n )$ is a commuting $n$-tuple of isometries with $\prod_{i=1}^n\widehat{V}_i=M_z$ on $H^2(\mathcal{D}_{V^*})$. Therefore, $(\widehat{V}_1,\dots ,\widehat{V}_n)$ is an isometric dilation of $(V_1,\dots ,V_n)$ with $M_z$ being the minimal isometric dilation of $V$. Therefore, by Theorem \ref{puredil}, there are unique commuting unitaries $U_1,\dots, U_n$ and unique orthogonal projections $P_1,\dots P_n$ in $\mathcal{B}(\mathcal{D}_{V^*})$ such that $\prod_{i=1}^nU_i=I$ and that the conditions $(1)-(3)$ are satisfied. 
	
	Conversely, suppose there are unique orthogonal projections $P_1,\ldots ,P_n$ and unique commuting unitaries $U_1,\ldots ,U_n$ in $\mathcal{B}(\mathcal{D}_{V^*})$ such that $\prod_{i=1}^n U_i=I$ and that the conditions $(1)-(3)$ are satisfied. Let $V_i=M_{U_iP_i^{\perp}+zU_iP_i}$ for each $i=1,2\dots ,n$. Then as seen in the proof of Theorem \ref{puredil}, $(V_1,\dots, V_n) $ is a commuting $n$-tuple of isometries with $\prod_{i=1}^nV_i=M_z$ on $H^2(\mathcal{D}_{V^*})$. 
\end{proof}

Also, Theorem \ref{BDF lemma O} provides a factorization of a $C._0$ isometry (i.e. a pure isometry) in terms of $n$ number of commuting isometries. Our result, Theorem \ref{puredil}, gives a factorization of a $C._0$ contraction in the following way.

\begin{thm} \label{thm:factors}
	Let $T_1,\ldots ,T_n$ be commuting contractions on a Hilbert space $\HS$ and let their product $T=\prod_{i=1}^n T_i$ be a $C._0$ contraction. Then $(T_1^*, \dots, T_n^*)\equiv (V_1^*|_{\mathcal{H}}, \dots , V_n^*|_{\mathcal{H}})$ for a model $n$-isometry $(V_1, \dots , V_n)$ on $H^2(\mathcal D_{T^*})$ if and only if there exist unique orthogonal projections $P_1, \dots , P_n$ and unique commuting unitaries $U_1, \dots , U_n$ in $\mathcal B(\mathcal D_{T^*})$ such that $\prod_{i=1}^n U_i=I_{\mathcal D_{T^*}}$ and the following conditions are satisfied:
	\begin{enumerate}
		\item $D_{T^*}T_i^*=P_i^{\perp}U_i^*D_{T^*}+P_iU_i^*D_{T^*}T^*$,
		\item $P_i^{\perp}U_i^*P_j^{\perp}U_j^*=P_j^{\perp}U_j^*P_i^{\perp}U_i^*$ ,
		\item $U_iP_iU_jP_j=U_jP_jU_iP_i$ ,
		\item  $P_1+U_1^*P_2U_1+U_1^*U_2^*P_3U_2U_1+\ldots +U_1^*U_2^*\ldots U_{n-1}^*P_nU_{n-1}\ldots U_2U_1 =I_{\mathcal{D}_T}$. 
	\end{enumerate}
\end{thm}
\begin{proof}
	First suppose there are projections $P_1, \dots , P_n$ and commuting unitaries $U_1, \dots , U_n$ in $\mathcal B(\mathcal D_{T^*})$ such that $\prod_{i=1}^n U_i=I_{\mathcal D_{T^*}}$ and that the conditions $(1)-(4)$ are satisfied. Then Theorem \ref{puredil} provides commuting $n$ isometries $V_1, \dots, V_n$ on $H^2(\mathcal D_{T^*})$ with $V_i=M_{U_iP_i^{\perp}+zU_iP_i}$ for each $i$, such that 
	\[
	(T_1^*,\dots ,T_n^*)\equiv (V_1^*|_{ W(\mathcal H)}, \dots , V_n^*|_{ W(\mathcal H)}),
	\]
	where $W$ is the isometry as in (\ref{eqn:0f1}).
	
	Conversely, suppose $(T_1^*,\ldots ,T_n^*)$ is equivalent to $(V_1^*|_{\mathcal{H}},\ldots,V_n^*|_{\mathcal{H}})$ for some model $n$ isometry $(V_1,\ldots ,V_n)$ on $H^2(\mathcal{D}_{T^*})$. Then by the $(\Rightarrow)$ part of Theorem \ref{puredil}, there are projections $P_1, \dots , P_n$ and commuting unitaries $U_1, \dots , U_n$ in $\mathcal B(\mathcal D_{T^*})$ such that $\prod_{i=1}^n U_i=I_{\mathcal D_{T^*}}$ and that the conditions $(1)-(4)$ are satisfied.
\end{proof}

\section{Examples} \label{Examples}

\vspace{0.2cm}

\noindent In this Section, we present several examples to compare our classes of commuting contractions admitting isometric dilation with the previously determined various classes from the literature. We shall show that neither our classes are properly contained in any of these classes from the literature nor any previously determined classes are contained properly in our classes. However, there are always nontrivial intersections. Note that our theory is one-dimensional in the sense that the operator models that we obtained are all having multiplicatiplication operators with multipliers of one complex variable.

\smallskip

Suppose $(T_1,T_2)$ is a commuting pair of contractions admitting isometric dilation $(V_1,V_2)$ on the minimal dilation space of $T=T_1T_2$. If there are unitaries $U_1,U_2$ and projections $P_1,P_2$ such that $U_iP_j=U_jP_i$ for $i=1,2$ and that the condition-$(1)$ of Theorem \ref{coromain} holds, then it can be verified that the conditions $(2),(3),(4)$ hold as a consequence. Hence, $(T_1,T_2)$ has an isometric dilation on the minimal ismetric dilation space of $T$. Now one may ask a question: if $(T_1,T_2)$ admits isometric dilation on the minimal isometric dilation space of $T$, then will the corresponding unitaries commute with the projections ? The following example gives a negative answer to this. 
\begin{eg}
	Let $T_1=T_2=\begin{bmatrix}
		0&1\\
		0&0
	\end{bmatrix}$ be commuting pair of contractions on $\mathbb{C}^2$. Then clearly $T=T_1T_2=0$ and hence $D_T=I$. Hence we need to find commuting $U_1,U_2$ and projections $P_1,P_2$ such that $D_TT_i=P_i^{\perp}U_i^*D_T+P_iU_i^*D_TT$ for $i=1,2$. Substituting $D_T=I$ and $T=0 $, the above equations are equivalent to $T_1=P_1^{\perp}U_1^* $ ad $T_2=P_2^{\perp}U_2^*$. One can observe that $P_1=P_2=\begin{bmatrix}
		0&0\\
		0&1
	\end{bmatrix}$ and $U_1=U_2=\begin{bmatrix}
		0&1\\1&0
	\end{bmatrix}$ satisfy the above two equations. It is clear that $U_1, U_2$ are commuting unitaries and $P_1, P_2$ are projections. Further as $T_1,T_2$ commute with each other, it follows that $P_1^{\perp}U_1^*$ commutes with $P_2^{\perp}U_2^*$. Simple calculation shows that, $U_1P_1=T_1$ and $U_2P_2=T_2$. Therefore $U_1P_1$ commutes with $U_2P_2$. further $D_TU_iP_iU_i^*D_T=U_iP_iU_i^*=U_iP_iP_iU_i^*=T_iT_i^*=\begin{bmatrix}
		1&0\\0&0
	\end{bmatrix}=D_{T_i}^2.$ Also $P_1+U_1^*P_2U_1=I$. Hence the conditions $(1)-(5)$ of Theorem \ref{main} hold. But one can clearly observe that $U_i$ do not commute with $P_j$ as $U_iP_j=\begin{bmatrix}
		0&1\\0&0
	\end{bmatrix}$ and $P_jU_i=\begin{bmatrix}
		0&0\\1&0
	\end{bmatrix}$. 
\end{eg}

Before going to the next example let us note down the following observation from the proofs of the previously stated dilation theorems.

\begin{note}\label{impnote}
	Let $T_1,\dots ,T_n\in \mathcal{B}(\mathcal{H})$ be commuting contractions. Suppose there are projections $P_1,\dots ,P_n\in \mathcal{B}(\mathcal{D}_T)$ and commuting unitaries $U_1,\dots ,U_n\in\mathcal{B}(\mathcal{D}_T)$ such that $\prod_{i=1}^nU_i=I$ and conditions $(1)-(5)$ of Theorem \ref{main} are satisfied. Then by the Theorem \ref{main}, $(T_1,\dots ,T_n)$ possesses an isometric dilation $(V_1,\dots ,V_n)$ on the minimal isometric dilation space $\mathcal{K}$ of $T$ such that $\prod_{i=1}^nV_i=V$ is the minimal isometric dilation of $T$. Without loss of generality we can assume $V_i$ to be as in (\ref{eqn:isom-01}) and $V$ to be the sch$\ddot{a}$ffer's minimal isometric dilation. So, if $V_i'=\prod_{j\neq i}V_j$, then 
	\[
	{V}_i'= V_i^*V=\begin{bmatrix}
		T_i'&0&0&0&\ldots \\
		U_iP_i^{\perp}D_T&U_iP_i&0&0&\ldots \\
		0&U_iP_i^{\perp}&U_iP_i&0&\ldots \\
		0&0&U_iP_i^{\perp}&U_iP_i&\ldots \\
		\ldots & \ldots & \ldots & \ldots & \ldots
	\end{bmatrix}.
	\]
	Thus, $(2,1)$ entries of both sides of ${V}_i'{V}={V}{V}_i' $ gives $D_TT_i'=U_iP_iD_T+U_iP_i^{\perp}D_TT$ for $1\leq i \leq n$.	 
\end{note}

Let $\mathcal{U}^n(\mathcal{H}) $ be the class of commuting $n$-tuples of contractions on $\mathcal{H}$ satisfying conditions $(1)-(4)$ of Theorem \ref{coromain} and let $\mathcal{S}^n(\mathcal{H})$ denote the class satisfying conditions $(1)-(5)$ of Theorem \ref{main}. The following example shows that $\mathcal{S}^n(\mathcal{H})$ is properly contained in $\mathcal{U}^n(\mathcal{H})$.
 
\begin{eg}\label{eg1}
	Let us consider the following doubly commuting contractions acting on $\C^3$:
	\[
	T_1=\begin{bmatrix}
		1&0&0\\0&0&0\\0&0&0
	\end{bmatrix} , \;
	T_2=\begin{bmatrix}
		0&0&0\\0&1&0\\0&0&0
	\end{bmatrix}, \;
	T_3=\begin{bmatrix}
		0&0&0\\0&0&0\\0&0&1
	\end{bmatrix}.
	\]
	We show that $(T_1,T_2,T_3)\notin \mathcal{S}^3(\mathbb{C}^3)$ though it belongs to $ \mathcal{U}^3(\mathbb{C}^3) $. Note that $T_1'=T_2T_3=0$ and similarly $T_2'=T_3'=0$. Also, it is clear that $T=T_1T_2T_3=0$ and hence $D_T=I$ on $\mathbb{C}^3$. Thus, $\mathcal{D}_T=\mathbb{C}^3$. Now suppose there are commuting unitaries $U_1,U_2,U_3$ and projections $P_1,P_2,P_3$ in $\mathcal{B}(\mathbb{C}^3)$ satisfying the hypotheses of Theorem \ref{main}. Then, for each $i$ we have that  $D_TT_i=P_i^{\perp}U_i^*D_T+P_iU_i^*D_TT$. Since $T=0$ and $D_T=I$ it reduces to $T_i=P_i^{\perp}U_i^*$. Again, from Note \ref{impnote}, it is clear that the unitaries and projections satisfying $(1)-(5)$ must also satisfy $D_TT_i'=U_iP_iD_T+U_i^*P_i^{\perp}D_TT$, which is same as saying that $T_i'=U_iP_i$. Thus, $T_i^*+T_i'=U_iP_i^{\perp}+U_iP_i=U_i$. Since $T_i'=0$, we have that $U_i=T_i^*+T_i'=T_i^*$. This contradicts the fact that $U_i$ is a unitary. Hence  $(T_1,T_2,T_3)\notin \mathcal{S}^3(\mathbb{C}^3)$. Now if we take $U_i=I$ and $P_i=I-T_i$  for $i=1,2,3$ then one can easily verify that the conditions $(1)-(4)$ of Theorem \ref{coromain} hold.  
\end{eg}

 In \cite{B:D:H:S}, Barik, Das, Haria and Sarkar introduced a new class of commuting contractions that admit isometric dilation. For each natural number $n\geq 3$ and for every number $p,q$ with $1\leq p<q \leq n$ let $\mathcal{T}^n_{p,q}(\mathcal{H})$ be defined as follows:
\begin{equation}\label{barikdasclass}
	 \mathcal{T}^n_{p,q}(\mathcal{H})=\{T\in \mathcal{T}^n(\mathcal{H}):\hat{T}_p,\hat{T}_q\in \mathbb{S}_{n-1}(\mathcal{H}),\,\text{ and } \hat{T}_p \text{ is pure} \}, \
\end{equation}
where for any natural number $i\leq n$, $\hat{T}_i=(T_1,\ldots ,T_{i-1},T_{i+1},\ldots T_n)$, $\mathcal{T}_{n}(\mathcal{H})$ is a set of commuting $n$ tuple of contractions on space $\mathcal{H}$ and
\[
\mathbb{S}_n(\mathcal{H})=\{T\in \mathcal{T}^n(\mathcal{H}):\sum_{k\in \{0,1\}^n}(-1)^{|k|}T^kT^{*k} \geq 0\} .
\]
Note that $\mathbb{S}_n(\mathcal H)$ is the set of all those $n$-tuples of contractions on $\mathcal{H}$ that satisfy Szego positivity condition, i.e. $\sum_{k\in \{0,1\}^n}(-1)^{|k|}T^kT^{*k} \geq 0$. The class obtained by putting an additional condition $\|T_i\|<1$ for each $i$ on the elements of $\mathcal{T}^n_{p,q}(\mathcal{H})$ is denoted by $\mathcal{P}^n_{p,q}(\mathcal{H})$. This class has been studied in \cite{G:K:V:V:W} by Grinshpan, Kaliuzhnyi,Verbovetskyi, Vinnikov and Woerdeman. In \cite{B:D:H:S}, it is shown that $\mathcal{P}^n_{p,q}(\mathcal{H})\subsetneq \mathcal{T}^n_{p,q}(\mathcal{H})$ for $1\leq p<q \leq n$. The following example shows that an element of our class may not satisfy the Szego positivity condition.

\begin{eg} \label{exmp:05}
	Let $T_1=T_2=\begin{bmatrix}
		0&0\\
		1&0
	\end{bmatrix}$ and $T_3=I$. Then $T=T_1T_2T_3=0$ and hence $D_{T}=I$. Let $P_1=P_2=\begin{bmatrix}
		1&0\\
		0&0
	\end{bmatrix}$, $P_3=0$ and $U_1=U_2= \begin{bmatrix}
		0&1\\
		1&0
	\end{bmatrix} $, $U_3=I$. Then $U_1,U_2,U_3$ are commuting unitaries and $P_1,P_2,P_3$ are projections satisfying $U_1U_2U_3=I$ and the conditions $(1)-(5)$ of Theorem \ref{main}. Hence $(T_1,T_2,T_3)$ belongs to $\mathcal{S}^3(\mathbb{C}^3)$. Note that
	\[
	I-T_1T_1^*-T_2T_2^*-T_3T_3^*=I-\begin{bmatrix}
		0&0\\
		0&1
	\end{bmatrix}-\begin{bmatrix}
		0&0\\
		0&1
	\end{bmatrix}-I=\begin{bmatrix}
		0&0\\
		0&-2
	\end{bmatrix} \ngeq 0.
	\]
Also,
\[
I-T_2T_2^*-T_3T_3^*=I-\begin{bmatrix}
		0&0\\
		0&1
	\end{bmatrix}-I= \begin{bmatrix}
		0&0\\
		0&-1
	\end{bmatrix} \ngeq 0
\]
and
\[
I-T_1T_1^*-T_3T_3^*=I-\begin{bmatrix}
		0&0\\
		0&1
	\end{bmatrix}-I= \begin{bmatrix}
		0&0\\
		0&-1
	\end{bmatrix} \ngeq 0.
\]
Thus, if we consider $\widehat{T}_1=(T_2,T_3)$, $\widehat{T}_2=(T_1,T_3)$ and $\widehat{T}_3=(T_1,T_2)$ then $\widehat{T}_1$, $\widehat{T}_2$ do not belong to $\mathbb{S}_{2}(\mathcal{H})$. So, for any $p,q$ satisfying $1\leq p,q \leq 3$, we have that $(T_1,T_2,T_3)\notin \mathcal{T}^n_{p,q}(\mathcal{H})$.   
\end{eg}
Note that Example \ref{exmp:05} does not satisfy Brehmer's condition. Recall that $\underline{T}=(T_1,\dots ,T_n)$ satisfies Brehmer's conditions if 
\[
\sum_{F\subseteq G}(-1)^{|F|}\underline{T}_{F}^*\underline{T}_F\geq 0
\] 
for all $G\subseteq\{ 1,\ldots n \}$. See (\ref{lit:eqn01}) for the definition. For $\gamma=(\gamma_1,\ldots ,\gamma_n)\in \mathbb{N}^n$, Curto and Vasilescu \cite{Cur:Vas 1}, \cite{Cur:Vas 2} introduced a notion of $\gamma$ contractivity. For $\gamma=(1,\ldots ,1):=e$ the notion agrees with Brehmer's condition. The notion of $\gamma$ contractivity is defined in such a way that for each $\gamma \in \mathbb N^n$ if an operator $T=(T_1,\ldots ,T_n)$ is $\gamma$ contractive then it is $e$ contractive. Since Example \ref{exmp:05} does not satisfy Brehmer's condition, it does not have a regular unitary dilation and it is not $e$ contractive. Hence, it can not be $\gamma$ contractive and consequently it does not belong to the `Curto-Vasilescu' class.
   
\begin{eg}
	As observed by Barik, Das, Haria and Sarkar in \cite{B:D:H:S} that there is an operator tuple $(M_{z_1},\ldots ,M_{z_n})$ in $\mathcal{T}^n_{p,q}(H^2(\mathbb{D}^n))$ that does not belong to $\mathcal{P}^n_{p,q}(H^2(\mathbb{D}^n))$. This class was introduced in \cite{G:K:V:V:W} by Grinshpan et al. Since 
	$T=M_{z_1}\ldots M_{z_n}$ on $H^2(\mathbb{D}^n)$ is an isometry, $\mathcal{D}_{T}=0$. So the minimal isometric dilation space of $T$ is  $\mathcal{K}_0=H^2(\mathbb{D}^n)$. Since each $M_{z_i}=T_i$ is an isometry on $H^2(\mathbb{D}^n)$, we have that $(M_{z_1},\ldots ,M_{z_n})$ is an isometric dilation of $(T_1,\ldots ,T_n)$ on the minimal dilation space of $T$ with the product being the minimal isometric dilation of $T$. Thus $(M_{z_1},\ldots ,M_{z_n})\in \mathcal{U}^n(\mathcal{H})$ by Theorem \ref{main}.   
\end{eg}

\begin{eg}\label{eg2}
Let us consider the following commuting self-adjoint scalar matrices acting on $\C^3$:
	\[
	T_1=\begin{bmatrix}
		0&0&0\\
		0&1&0\\
		0&0&1
	\end{bmatrix}, \;
	T_2=\begin{bmatrix}
		1&0&0\\
		0&0&0\\
		0&0&1
	\end{bmatrix}, \;
	T_3=\begin{bmatrix}
		1&0&0\\
		0&1&0\\
		0&0&0
	\end{bmatrix}.
	\]
Let $P_i=I-T_i$ and $U_i=I$ for $i=1,2,3$. Note that $D_T=0$ and the projections $P_1,P_2,P_3$ satisfy $P_1+P_2+P_3=I$. Thus, the condition-$(5)$ holds. Also, it can be easily verified that the conditions $(1)-(4)$ hold. Therefore, this triple of commuting contractions belongs to $\mathcal{S}^3(\mathbb{C}^3)$.
 \end{eg}

\begin{eg}\label{eg3}
	Let us consider the commuting self-adjoint scalar matrices
	\[
	T_1=\begin{bmatrix}
		0&1/3&0\\
		1/3&0&0\\
		0&0&0
	\end{bmatrix}, \; T_2=\begin{bmatrix}
		1/2&1/3&0\\
		-1/3&1/2&0\\
		0&0&0
	\end{bmatrix}, \; T_3=\begin{bmatrix}
		0&0&0\\
		0&0&0\\
		0&0&1
	\end{bmatrix}
	\]
	acting on $\mathbb{C}^3$. Evidently $T=T_1T_2T_3=0$ and thus $D_T=I$. Now suppose there are projections $P_1,P_2,P_3$ and commuting unitaries $U_1,U_2,U_3$ satisfying conditions $(1)-(4)$ of Theorem \ref{main}. So, in particular we have
	$$D_TT_1=P_1^{\perp}U_1^*D_T+P_1U_1^*D_TT,
	$$
	which implies that $T_1=P_1^{\perp}U_1^*$. We also have $D_TU_1P_1U_1^*D_T=I-T_1^*T_1$ and this gives $P_1=U_1^*(I-T_1^*T_1)U_1$ and hence we have
	\[
	P_1^{\perp}=I-U_1^*(I-T_1^*T_1)U_1=U_1^*T_1^*T_1U_1=U_1^*\begin{bmatrix}
		1/9&0&0\\0&1/9&0\\0&0&0
	\end{bmatrix}U_1.
	\]
	Therefore,
	$$T_1=U_1^*\begin{bmatrix}
		1/9&0&0\\0&1/9&0\\0&0&0
	\end{bmatrix}.
	$$
	Again
	\[
	\begin{bmatrix}
		3\\0\\0
	\end{bmatrix}=T_1\begin{bmatrix}
		0\\9\\0
	\end{bmatrix}=U_1^*\begin{bmatrix}
		1/9&0&0\\0&1/9&0\\0&0&0
	\end{bmatrix}\begin{bmatrix}
		0\\9\\0
	\end{bmatrix}=U_1^*\begin{bmatrix}
		0\\1\\0
	\end{bmatrix}.
	\]
	This contradicts the fact that $U_1$ is a unitary. Hence $(T_1,T_2,T_3)\notin \mathcal{U}^3(\mathbb{C}^3)$. 
\end{eg} 

\begin{eg}
	Let $T_1,T_2$ be as in Example \ref{exmp:06} and let $T_3=I$. Then it i evident that $(T_1,T_2,T_3)$ does not belong to our class $\mathcal{S}^3(\mathbb{C}^3)$. This can be verified using an argument similar to that in Example \ref{eg3}. However,
	\[
	I-T_2T_2^*-T_3T_3^*+T_2T_3T_2^*T_3^*=I-\begin{bmatrix}
		0&0&0\\
		0&0&0\\
		0&0&1/3
	\end{bmatrix}-I+\begin{bmatrix}
		0&0&0\\
		0&0&0\\
		0&0&1/3
	\end{bmatrix}=0.
	\]
	Alo,
	\[
	I-T_1T_1^*-T_2T_2^*+T_1T_2T_1^*T_2^*=I-\begin{bmatrix}
		0&0&0\\
		0&0&0\\
		0&0&1/3
	\end{bmatrix}-\begin{bmatrix}
		0&0&0\\
		0&1/9&0\\
		0&0&1/27
	\end{bmatrix}+0=\begin{bmatrix}
		1&0&0\\
		0&8/9&0\\
		0&0&17/27
	\end{bmatrix}\geq 0.
	\]
	Hence $\widehat{T}_1$ and $\widehat{T}_3$ belong to $\mathbb{S}_3(\mathbb{C}^3)$. Clearly $T_1,T_2$ are pure contractions and hence $\widehat{T}_3$ is pure. Thus, $(T_1,T_2,T_3)\in \mathcal{T}^3_{1,3}(\mathbb{C}^3)$.    
\end{eg}

We now consider a few examples of contractions acting on infinite dimensional Hilbert spaces and study their dilations.

\begin{eg}
	Let $\mathcal{H}=l^2$ and let $T_1,T_2$ be the following weighted shift operators acting on $\HS$: 
	\[T_1(h_0,h_1,\dots )= (0,a_0h_0,0,a_1h_2,0,a_2h_4,\dots),\]
	\[T_2(h_0,h_1,\dots )= (0,b_0h_0,0,b_1h_2,0,b_2h_4,\dots),\]
	where $\{a_0,a_1,\dots\}$ and $\{b_0,b_1,\dots \}$ are  bounded sequences of non zero real numbers with $|a_n|<1$, $|b_n|<1$ for every $n\in \mathbb N$. Clearly, $T=T_1T_2=T_2T_1=0$. Hence $D_T=I$ and $\mathcal{D}_T=l^2=\mathcal{H}$.   
	We show that $(T_1,T_2)$ does not possess isometric dilation $(V_1,V_2)$ on the minimal isometric dilation space of $T$ with $V=V_1V_2 $ being the minimal isometric dilation of $T$.	Let us assume the contrary. Then by Theorem \ref{main}, there are projections $P_1,P_2$ commuting unitaries $U_1,U_2$ in $\mathcal{B}(\mathcal{D}_T)=\mathcal{B}(\mathcal{H})$ with $U_1U_2=I$ satisfying conditions $(1)-(5)$. Condition $(1)$ gives us $T_i=P_i^{\perp}U_i^*$ for $i=1,2$.  From condition $(4)$ we have
	\begin{align*}
		&U_1P_1U_1^*=I-T_1^*T_1\\
		\implies &T_1^*T_1=U_1(I-P_1)U_1^*\\
		\implies &T_1^*T_1=U_1T_1 \hsp [\text{as } T_1=P_1^{\perp}U_1^* ].
	\end{align*}  
Simple calculation shows that $T_1^*T_1=diag(a_0^2,0,a_1^2,0,a_2^2,\dots)$. Hence \begin{align*}
	&U_1T_1(a_0^{-1},0,\dots)=T_1^*T_1(a_0^{-1},0,\dots)\\
	\implies & U_1(0,1,0,0,\dots)=(a_0,0,0,\dots).
\end{align*}
Since $U_1$ is a unitary, $\|U_1(0,1,0,0,\dots)\|=\|(0,1,0,0,\dots)\| $ and thus it follows from here that $\| (a_0,0,0,\dots)\|=\|(0,1,0,0,\dots) \|$. This is a contradiction as $|a_0|<1$. Hence we are done.
\end{eg}

\begin{eg}
Let $\mathcal{H}=l^2$ and let $\{e_0,e_1,\dots \}$ be the standard orthonormal basis of $l^2$. For $k=1,\dots ,n $, let $T_k$ be defined on $\HS$ by $T_k(e_m)=0 $ if $ m\equiv (k-1)mod \,n$ and $T_k(e_m)=e_m$ otherwise. We show that $(T_1,T_2,\dots ,T_n)$ possesses an isometric dilation on the minimal isometric dilation space $\mathcal{K}$ of $T=\prod_{i=1}^nT_i$. evidently, $T=0$ and thus $\mathcal{D}_T= \mathcal{H}$. So, we have that $\mathcal{K}=\mathcal{H}\oplus \mathcal{H}\oplus \dots$ is the minimal isometric dilation space of $T$. Now for each $k$, define $P_k=I-T_k$ and $U_k=I $ on $\mathcal{D}_T= \mathcal{H}$. Then it follows that $U_1,\dots ,U_n$ are commuting unitaries with $ U_1U_2\dots U_n=I$ and $P_1,\dots ,P_n$ are projections. Also, the conditions $(1)-(5)$ of Theorem \ref{main} are satisfied straightway. Therefore, Theorem \ref{main} guaranties the existence of an isometric dilation of $(T_1, \dots , T_n)$ as desired.      
\end{eg}

\vspace{0.1cm}

\section{Sz. Nagy-Foias type isometric dilation and functional model}

\vspace{0.2cm}

\noindent Recall that a c.n.u. contraction $T$ on a Hilbert space $\HS$ is a contraction such that there is no non-zero subspace $\HS_1$ of $\HS$ that reduces $T$ and on ${\HS_1}$ the operator $T$ acts as a unitary. In simple words, a c.n.u. contraction is a contraction without any unitary part. Let $T$ on $\mathcal{H}$ be a c.n.u. contraction and let $V$ on $\mathcal{K}_{0}$ be the minimal isometric dilation of $T$. By Wold decomposition $\mathcal K_0$ splits into reducing subspaces $\mathcal{K}_{01}$, $\mathcal{K}_{02}$ of $V$ such that $\mathcal{K}_0=\mathcal{K}_{01}\oplus \mathcal{K}_{02}$ and that $V|_{\mathcal{K}_{01}}$ is unitarily equivalent to a unilateral shift and $V|_{\mathcal{K}_{02}}$ is a unitary. Then $\mathcal{K}_{01}$ can be identified with $H^2(\mathcal D_{T^*})$ and $\mathcal{K}_{02}$ can be identified with $\ov{\Delta_{T}(L^2(\mathcal{D}_T))}$, where $\Delta_{T}(t)=[I_{\mathcal D_T}-\Theta_T(e^{it})^*\Theta_T(e^{it})]^{1/2}$ and $\Theta_T$ being the characteristic function of the contraction $T$. For further details see Chapter-VI of \cite{Nagy}. Thus, $\mathcal{K}_{0}= \mathcal{K}_{01}\oplus \mathcal{K}_{02}$ can be identified with $\mathbb{K}_{+}=H^2(\mathcal D_{T^*})\oplus \overline{\Delta_{T}(L^2(\mathcal D_T))}$. Also, the isometry $V$ on $\mathcal{K}_0$ can be realized as $M_z \oplus M_{e^{it}}|_{\overline{\Delta_T(L^2(\mathcal D_T))}}$. Therefore, there is a unitary
\begin{equation}\label{unitary between to spaces}
	\tau=\tau_1\oplus \tau_2:\mathcal{K}_{01}\oplus \mathcal{K}_{02}\to (H^2\otimes \mathcal{D}_{T^*})\oplus \overline{\Delta_{T}(L^2(\mathcal{D}_T))}\, :=\widetilde{\mathbb{K}}_{+}
\end{equation}
such that $V$ on $\mathcal{K}_0$ can be realized as $(M_z \otimes I_{\mathcal D_{T^*}}) \oplus M_{e^{it}}|_{\overline{\Delta_T(L^2(\mathcal D_T))}}$ on $\widetilde{\mathbb{K}}_{+}$.

\smallskip

If $(T_1,\ldots ,T_n)$ is a tuple of commuting contractions acting on $\mathcal{H}$ satisfying the hypotheses of Theorem \ref{coromain}, then it possesses an isometric dilation $(V_1,\ldots ,V_n)$ on the minimal isometric dilation space $\mathcal{K}_0$ of $T$. By Wold decomposition of commuting isometries, we have that $\mathcal{K}_{01}$ and $\mathcal{K}_{02}$ are reducing subspaces for each $V_i$ and that \begin{equation}\label{infoV_i2}
V_{i2}=V_i|_{\mathcal{K}_{02}}
\end{equation}
is a unitary for $1\leq i \leq n$. Now we state our dilation theorem and functional model in the Sz. Nagy-Foias setting. This is another main result of this paper.
 
\begin{thm}\label{Nagy isodil}
	Let $(T_1,\ldots ,T_n)$ be a tuple of commuting contractions acting on $\mathcal{H}$ such that $T=\Pi_{i=1}^nT_i$ is a c.n.u. contraction. Suppose there are orthogonal projections $P_1,\ldots  ,P_n$ and commuting unitaries $U_1,\ldots ,U_n $ in $\mathcal{B}(\mathcal{D}_T)$ satisfying
	\begin{enumerate}
		\item $D_TT_i=P_i^{\perp}U_i^*D_T+P_iU_i^*D_TT$,
		\item  $P_i^{\perp}U_i^*P_j^{\perp}U_j^*=P_j^{\perp}U_j^*P_i^{\perp}U_i^*$,
		\item $U_iP_iU_jP_j=U_jP_jU_iP_i$ ,
		\item $D_TU_iP_iU_i^*D_T=D_{T_i}^2$ ,
	\end{enumerate} for $1\leq i<j\leq n$. 
	Then there are projections $Q_1,\ldots ,Q_n$ and commuting unitaries $\widetilde{U}_1,\ldots,\widetilde{U}_n$ in $\mathcal{B}(\mathcal{D}_{T^*})$ such that $(T_1, \dots , T_n)$ dilates to the tuple of commuting isometries $(\widetilde{V}_{11}\oplus \widetilde{V}_{12},\ldots ,\widetilde{V}_{n1}\oplus \widetilde{V}_{n2})$ on $\widetilde{\mathbb{K}}_+=H^2\otimes \mathcal{D}_{T^*}\oplus \ov{\Delta_{T}(L^2(\mathcal{D}_T))}$, where
	\begin{align*}
		\widetilde{V}_{i1}&=I\otimes\widetilde{U}_iQ_i^{\perp}+M_z\otimes \widetilde{U}_iQ_i\,,\\
		\widetilde{V}_{i2}&=\tau_2V_{i2}\tau_2^* \,,
	\end{align*} for unitaries $\tau_2$ and $V_{i2}$ as in \eqref{unitary between to spaces} and \eqref{infoV_i2} respectively for $1 \leq i \leq n$.
	 
\end{thm}

\begin{proof}
By Theorem \ref{coromain}, we have that $(T_1,\ldots ,T_n)$ possesses an isometric dilation $(V_1,\ldots ,V_n)$ on $\mathcal{K}_0=\mathcal{H}\oplus l^2(\mathcal{D}_{T})$ which in fact satisfies $V_i^*|_{\mathcal{H}}=T_i^*$ for $1\leq i \leq n$. Now $\mathcal{K}_0$ has an orthogonal decomposition $\mathcal{K}_0=\mathcal{K}_{01}\oplus \mathcal{K}_{02}$ such that $\mathcal{K}_{01}$ and $\mathcal{K}_{02}$ are common reducing subspaces for $V_1,\ldots ,V_n$, $(V_1|_{\mathcal{K}_{01}},\ldots ,V_n|_{\mathcal{K}_{01}})$ is a pure isometric tuple, i.e. $\prod_{i=}^n V_i|_{\mathcal{K}_{01}}$ is a pure isometry and $(V_1|_{\mathcal{K}_{02}},\ldots  ,V_n|_{\mathcal{K}_{02}})$ is a unitary tuple.
Let $ \prod_{i=1}^nV_i=V$, $V_i|_{\mathcal{K}_{01}}=V_{i1}$ and $V_i|_{\mathcal{K}_{02}}=V_{i2}$ for $1\leq i \leq n$.  Also, let $ \prod_{i=1}^nV_{i1}=W_1$ and $ \prod_{i=1}^nV_{i2}=W_2 $. So $W_1$ on $\mathcal{K}_{01}$ is a pure isometry and $W_2$ on $\mathcal{K}_{02}$ is a unitary. So,
$$
(\tau V_1 \tau^*,\ldots ,\tau V_n \tau^*)= (\tau_1 V_{11} \tau_1^*\oplus \tau_2 V_{12} \tau_2^*,\ldots ,\tau_1 V_{n1} \tau_1^*\oplus \tau_2 V_{n2} \tau_2^*)
$$
is an isometric dilation of $(T_1,\ldots , T_n)$ on $\widetilde{\mathbb{K}}_+$, where $\tau=\tau_1\oplus \tau_2$ is as in \eqref{unitary between to spaces}. Thus, the tuple $(\tau_1 V_{11} \tau_1^*,\ldots \tau_1 V_{n1} \tau_1^*)$ on $H^2\otimes \mathcal{D}_{T^*}$ is a pure isometric tuple. Hence, by Theorem \ref{BCL}, there exist commuting unitaries $\widetilde{U}_1,\ldots , \widetilde{U}_n$ and orthogonal projections $Q_1,\ldots ,Q_n$ in $\mathcal{B}(\mathcal{D}_{T^*}) \; (=\mathcal B(D_{W_1^*}))$ such that the tuple $(\tau_1 V_{11} \tau_1^*,\ldots \tau_1 V_{n1} \tau_1^*)$ is unitarily equivalent to 
$$
(I\otimes \widetilde{U}_1Q_1^{\perp}+M_z \otimes \widetilde{U}_1Q_1,\ldots ,I\otimes \widetilde{U}_nQ_n^{\perp}+M_z \otimes \widetilde{U}_nQ_n) \quad \text{ on } \quad {H}^2\otimes \mathcal{D}_{T^*}
$$
via a unitary, say $Z$. For $1\leq i \leq n$, let us denote  $\widetilde{V}_{i1}= I\otimes\widetilde{U}_iQ_i^{\perp}+M_z \otimes \widetilde{U}_iQ_i$ and $\widetilde{V}_{i2}=\tau_2 V_{i2} \tau_2^* $. So, $(\tau V_1 \tau^*,\ldots ,\tau V_n \tau^*) $ is unitarily equivalent to $(\widetilde{V}_{11}\oplus \widetilde{V}_{12},\ldots , \widetilde{V}_{n1}\oplus \widetilde{V}_{n2}) $ via the unitary $Z\oplus I$. Thus, $(V_1,\ldots ,V_n)$ is unitarily equivalent to $(\widetilde{V}_{11}\oplus \widetilde{V}_{12},\ldots , \widetilde{V}_{n1}\oplus \widetilde{V}_{n2})$ via a unitary \[
Y=Z\tau_1\oplus \tau_2:\mathcal{K}_{01}\oplus \mathcal{K}_{02}\to H^2\otimes \mathcal{D}_{T^*}\oplus \ov{\Delta_{T}(L^2(\mathcal{D}_T))}.
\]
Let $\widetilde{V}_i=\widetilde{V}_{i1}\oplus \widetilde{V}_{i2}$, for $i=1,\ldots n$. Since for each $i$, $V_i^*|_{\mathcal{H}}=T_i^*$, we have for any $h\in \mathcal{H} $,
\[
\widetilde{V}_i^*(Yh)=(YV_iY^*)^*Yh=YV_i^*Y^*Yh=YV_i^*h=YT_i^*h. 
\]
Therefore, $\widetilde{V}_i^*|_{Y(\mathcal{H})}=YT_i^*Y^*|_{Y(\mathcal{H})}$ for $1\leq i \leq n$. So, we have $T_1^{k_1}\ldots T_n^{k_n}=Y^*\widetilde{V}_1^{k_1}\ldots \widetilde{V}_n^{k_n}Y $. Thus, $ (\widetilde{V}_{11}\oplus \widetilde{V}_{12},\ldots , \widetilde{V}_{n1}\oplus \widetilde{V}_{n2})$ is an isometric dilation of $(T_1,\ldots, T_n ) $, where $\widetilde{V}_{i1}= I\otimes\widetilde{U}_iQ_i^{\perp}+M_z \otimes \widetilde{U}_iQ_i$ and $\widetilde{V}_{i2}=\tau_2 V_{i2} \tau_2^* $ for $1\leq i \leq n$. This completes the proof. 
\end{proof}

\vspace{0.1cm}

\section{A model theory for a class of commuting contractions} \label{model-theory}

\vspace{0.2cm}

\noindent In this Section, we present a model theory for a tuple of commuting contractions satisfying the conditions of Theorems \ref{main} \& \ref{coromain}.

\begin{thm}\label{model}
	Let $(T_1,\ldots, T_n)$ be commuting tuple of contractions on a Hilbert space $\mathcal{H}$ and let $T=\Pi_{i=1}^nT_i$ . Suppose there are projections $P_1,\ldots ,P_n\in \mathcal{B}(\mathcal{D}_{T^*})$ and commuting unitaries $U_1,\ldots ,U_n\in \mathcal{B}(\mathcal{D}_{T^*})$ such that for each $i=1, \dots, n$,
	\begin{enumerate} 
		\item $D_{T^*}T_i^*=P_i^{\perp}U_i^*D_{T^*}+P_iU_i^*D_{T^*}T^*$,
		\item  $P_i^{\perp}U_i^*P_j^{\perp}U_j^*=P_j^{\perp}U_j^*P_i^{\perp}U_i^*$,
		\item $U_iP_iU_jP_j=U_jP_jU_iP_i$ ,
		\item $D_{T^*}U_iP_iU_i^*D_{T^*}=D_{T_i^*}^2$ .
	\end{enumerate} 
	Let $Z_1,\ldots ,Z_n$ on $\mathcal{K}=\mathcal{H}\oplus l^2(\mathcal{D}_{T^*})$ be defined as follows:
	\[ Z_i=\begin{bmatrix}
		T_i&D_{T^*}U_iP_i&0&0&\ldots \\
		0&U_iP_i^{\perp}&U_iP_i&0&\ldots \\
		0&0&U_iP_i^{\perp}&U_iP_i&\ldots \\
		0&0&0&U_iP_i^{\perp}&\ldots \\
		\ldots & \ldots & \ldots & \ldots & \ldots
	\end{bmatrix} , \qquad (1\leq i \leq n).
	\]
	Then, 
	\begin{enumerate}
		\item[(i)] $(Z_1,\ldots ,Z_n)$ is a commuting $n$-tuple of co-isometries, $\mathcal{H}$ is a common invarient subspace of $Z_1,\ldots ,Z_n$ and $Z_i|_{\mathcal{H}}=T_i$.
		\item[(ii)] There is an orthogonal decomposition $\mathcal{K}=\mathcal{K}_1\oplus \mathcal{K}_2$ into common reducing subspaces of $Z_1,\ldots ,Z_n$ such that $(Z_1|_{\mathcal{K}_1},\ldots ,Z_n|_{\mathcal{K}_1})$ is a pure co-isometric tuple, that is $Z|_{\mathcal{K}_1}=\Pi_{i=1}^nZ_i|_{\mathcal{K}_1}$ is a pure co-isometry and $(Z_1|_{\mathcal{K}_2},\ldots ,Z_n|_{\mathcal{K}_2})$ is a unitary tuple.
		\end{enumerate}
Additionally, if
\[
(5) \;P_1+U_1^*P_2U_1+U_1^*U_2^*P_3U_2U_1+\ldots +U_1^*U_2^*\ldots U_{n-1}^*P_nU_{n-1}\ldots U_2U_1=I_{\mathcal{D}_T} ,
\]
then
	\begin{enumerate}
		\item[(iii)] $\mathcal{K}_1$ can be identified with $H^2(\mathcal{D}_{Z})$, where $\mathcal{D}_Z$ has same dimension as that of $\mathcal{D}_T$. Also, there are projections $\widehat{P_1},\ldots ,\widehat{P_n}$, commuting unitaries $ \widehat{U_1},\ldots ,\widehat{U_n}$ such that the operator tuple $(Z_1|_{\mathcal{K}_1},\ldots, Z_n|_{\mathcal{K}_1})$ is unitarily equivalent to the multiplication operator tuple
		\[
		\left( M_{\widehat{U_1}\widehat{P_1}^{\perp}+\widehat{U_1}\widehat{P_1}\bar{z}},\ldots ,M_{\widehat{U_n}\widehat{P_n}^{\perp}+\widehat{U_n}\widehat{P_n}\bar{z}} \right)
		\]
		acting on $H^2(\mathcal{D}_{T})$.  
	\end{enumerate}
\end{thm}
\begin{proof}
	We apply Theorem \ref{coromain} to the tuple $(T_1^*,\ldots ,T_n^*)$ of commuting contractions to have an isometric dilation $(X_1,\ldots ,X_n)$ on $\mathcal{K}_0=\mathcal{H}\oplus l^2(\mathcal{D}_{T^*})$, where
	\[
	X_i=\begin{bmatrix}
		T_i^*&0&0&0&\ldots \\
		P_iU_i^*D_{T^*}&P_i^{\perp}U_i^*&0&0&\ldots \\
		0&P_iU_i^*&P_i^{\perp}U_i^*&0&\ldots \\
		0&0&P_iU_i^*&P_i^{\perp}U_i^*&\ldots \\
		\ldots & \ldots & \ldots & \ldots & \ldots
	\end{bmatrix} ,  \qquad (1\leq i \leq n). 
	\]
Clearly $Z_i=X_i^*$, for each $i$ and it is evident from the block-matrix that $\mathcal{H}$ is a common invariant subspace for each $Z_i$ and $Z_i|_{\mathcal{H}}=T_i$. This proves $(i)$. Since $(X_1,\ldots ,X_n)$ is a commuting tuple of isometry, $X=\Pi_{i=1}^nX_i$ is an isometry. By Wold decomposition, $\mathcal{K}=\mathcal{K}_1\oplus \mathcal{K}_2$ and that $X|_{\mathcal{K}_1}$ is a pure isometry, $X|_{\mathcal{K}_2}$ is a unitary. Also, they are common reducing subspaces for each $X_i$. Indeed, if
\[
X_i=\begin{pmatrix}
		A_i&B_i\\C_i&D_i
	\end{pmatrix} \quad \& \quad X=\begin{pmatrix}
		X_{K1} & 0\\
		0 & X_{K2}
	\end{pmatrix}
	\]
	with respect to the decomposition $\mathcal{K}=\mathcal{K}_1\oplus \mathcal{K}_2$ then $X_iX=XX_i$ implies that 
	\[
	B_iX_{K2}=X_{K1}B_i \quad \& \quad C_iX_{K1}=X_{K2}C_i.\]
Therefore, for all $k\in\mathbb{N}$ we have
$$
X_{K2}^{*k}B_i^*=B_i^*X_{K1}^{*k} \quad \& \quad X_{K1}^{*n}C_i^*=C_i^*X_{K2}^{*k}.
$$
Now $X_{K1}$ is a pure isometry and $X_{K2}$ is unitary, so on one hand we have $\|X_{K2}^{*k}B_i^*\|= \|B_i^*\|$ and on the other hand $\|B_i^*X_{K1}^{*k}\|\to 0$ as $k \rightarrow \infty $. Hence $B_i=0$. Similarly, $C_i=0$ for $1\leq i \leq n$. Thus, with respect to the decomposition $\mathcal{K}=\mathcal{K}_1\oplus\mathcal{K}_2$, the operator $X_i$ takes form
\[
X_i= \begin{pmatrix}
		X_{i1}&0\\
		0&X_{i2}
	\end{pmatrix}, \qquad (1\leq i \leq n).
	\]
	Now since $X_i$ is an isometry, so will be $X_{i1}$, $X_{i1}'=\Pi_{j\neq i}X_{j1}$ and $X_{i2}$, $X_{i2}'=\Pi_{j\neq i}X_{j2}$. Further from the block-matrix of $X_i$ and from the fact that $X=\Pi_{i=1}^nX_i$, it is clear that $X_{K2}=\Pi_{i=1}^nX_{i2}$. Again, $X_{K2}$ is a unitary, $X_{i2}'$ is an isometry and $X_{i2'}^*X_{K2}=X_{i2}$. So, we have
	\[
		X_{i2}X_{i2}^*=X_{i2'}^*X_{K2}X_{K2}^*X_{i2'}=I, \qquad (1\leq i \leq n).
	\]
	Thus, $X_{i2}$ is a unitary on $\mathcal{K}_2$ for each $i=1,\ldots ,n$. Further $(X_{11},\ldots,X_{n1} )$ is a pure isometric tuple as $X_{K1}=\Pi_{i=1}^nX_{i1}$ is a pure isometry. Since $Z_i=X_i^*$, we have that $(Z_1|_{\mathcal K_1},\ldots ,Z_n|_{\mathcal K_1})$ is pure co-isometric tuple and $(Z_1|_{\mathcal{K}_2},\ldots ,Z_n|_{\mathcal{K}_2})$ is a unitary tuple. Hence, $(ii)$ holds.  
	
	Additionally, if $(5)$ holds then the dilation $(X_1,\ldots ,X_n)$ satisfies the condition that $X$ is the Sch$\ddot{a}$ffer's minimal isometric dilation of $T^*$ by Theorem \ref{main}. Let us denote $Z=X^*$. Then $Z$ is a co-isometry as $X$ is an isometry and
	\[
	Z=\begin{bmatrix}
		T&D_{T^*}&0&0&\ldots\\
		0&0&I_{\mathcal{D}_{T^*}}&0&\ldots\\
		0&0&0&I_{\mathcal{D}_{T^*}}&\ldots\\
		\vdots&\vdots&\vdots&\vdots&\ldots
	\end{bmatrix}.
	\]
Note that the dimensions of $\mathcal{D}_{Z}$ and $\mathcal{D}_{T}$ are same. Indeed, if $\tau:\mathcal{D}_{T}\to \mathcal{D}_{Z}$ is defined by $\tau D_{T}h=D_{Z}h$ for all $h\in \mathcal{H}$ and extended continuously to the closure, then $\tau$ is a unitary. We recall the proof here. Since $X$ is the minimal isometric dilation of $T^*$, we have
\[
		\mathcal{K} =\overline{span}\{X^kh:k\geq 0,\,h\in \mathcal{H} \} =\overline{span}\{Z^{*k}h:k\geq 0,\,h\in \mathcal{H} \}.
	\]
	Now for $n\in \mathbb{N}$ and $h\in \mathcal{H}$, we have
	$$
	D_{Z}^2 Z^{*n}h=(I-Z^*Z)Z^{*n}h=Z^{*n}-Z^*Z^nZ^*h=0.
	$$
	Therefore, $D_{Z}X^nh=0$ for any $n\in \mathbb{N}$ and $h\in \mathcal{H}$. So, $\mathcal{D}_{Z}=\overline{D_{Z}\mathcal{K}}=\overline{D_{Z}\mathcal{H}}.$ Also,
	\[
		\|D_{Z}h\|^2=\lbrace(I-Z^*Z)h,h \rbrace=\|h\|^2-\|Zh\|^2=\|h\|^2-\|Th\|^2=\|D_{T}h\|^2.
	\]
	Therefore, $\tau$ is a unitary.  By Theorem \ref{BCL}, we have that
	$$
	(X_{11},\ldots ,X_{n1})\cong (M_{\widetilde{U}_1Q_1^{\perp}+z\widetilde{U}_1Q_1},\ldots , M_{\widetilde{U}_nQ_n^{\perp}+z\widetilde{U}_nQ_n}),
	$$ where $Q_1,\ldots ,Q_n$ are projections and $\widetilde{U}_1,\ldots \widetilde{U}_n$ are commuting unitaries from $\mathcal{B}(\mathcal{D}_{X_{K1}^*})$ satisfying
	\begin{equation}\label{X1ionK1} 
	D_{X_{i1}'^*}^2X_{i1}^*=D_{X_{K1}^*}Q_i^{\perp}\widetilde{U}_i^*D_{X_{K1}^*} 
	\end{equation}
	and 
	\begin{equation}\label{X1i'onK1} 
	D_{X_{i1}^*}^2X_{i1}'^*=D_{X_{K1}^*}\widetilde{U}_iQ_iD_{X_{K2}^*} \end{equation} for all $i=1, \dots, n$. Using the fact that $X_{i2},\ldots , X_{n2}$ are unitaries on $\mathcal{K}_2$, it follows that \[D_{X_i'^*}^2=I_{\mathcal{K}}-\begin{bmatrix}
		X_{i1}'&0\\
		0&X_{i2}'
	\end{bmatrix}\begin{bmatrix}
		X_{i1}'^*&0\\
		0&X_{i2}'^*
	\end{bmatrix}=\begin{bmatrix}
		I_{\mathcal{K}_1}-X_{i1}'X_{i1}'^*&0\\
		0&I_{\mathcal{K}_2}-X_{i2}'X_{i2}'^*
	\end{bmatrix}=\begin{bmatrix}
		I_{\mathcal{K}_1}-X_{i1}'X_{i1}'^*&0\\
		0&0
	\end{bmatrix}.
	\]
	Therefore, $D_{X_{i}'^*}=D_{X_{i1}'^*}\oplus 0$. Similarly we can prove that $D_{X_{i}^*}= D_{X_{i1}^*}\oplus 0$ for $1\leq i \leq n$, with respect to the above decomposition of $\mathcal{K}$. So, $D_{X^*}= D_{X_{K1}^*}\oplus0$. Hence $\mathcal{D}_{X_{K1}^*}=\mathcal{D}_{X^*}=\mathcal{D}_{Z}$. Let us denote $\widehat{U_i}=\tau^*\widetilde{U}_i\tau $ and $Q_i=\widehat{P_i}$ for $1\leq i\leq n$. Thus, $\widehat{U_1},\ldots, \widehat{U_n}$ are commuting unitaries and $\widehat{P_1},\ldots ,\widehat{P_n}$ are projections in $\mathcal{B}(\mathcal{D}(T))$ such that $(Z_1|_{\mathcal{K}_1},\ldots Z_n|_{\mathcal{K}_1})$ is unitarily equivalent to $\left( M_{\widetilde{U}_1Q_1^{\perp}+z\widetilde{U}_1Q_1},\ldots , M_{\widetilde{U}_nQ_n^{\perp}+z\widetilde{U}_nQ_n} \right)$, which can be realized as $\left( M_{\widehat{U_1}\widehat{P_1}^{\perp}+\widehat{U_1}\widehat{P_1}\bar{z}},\ldots ,M_{\widehat{U_n}\widehat{P_n}^{\perp}+\widehat{U_n}\widehat{P_n}\bar{z}} \right) $ on $H^2(\mathcal{D}_{T})$ via the unitary $\tau$. This proves $(iii)$ and the proof is complete.
	
\end{proof}

Apart from having the explicit constructions of isometric dilations and functional model for a commuting contractive tuple, another interesting consequence of Theorem \ref{main} is that it gives a commutant lifting in several variables as discussed in Remark \ref{comm:lift1}. We conclude this article here. There will be two more articles in this direction as sequels. One of them will describe explicit constructions of minimal unitary dilations of commuting contractions $(T_1, \dots , T_n)$ on the minimal unitary dilation spaces of $T=\prod_{i=1}^n T_i$. The other article will deal with dilations when the defect spaces $\mathcal D_T, \mathcal D_{T^*}$ are finite dimensional and their interplay with distinguished varieties in the polydisc.

\vspace{0.3cm}

\section{Data availability statement}

 \begin{enumerate}
 
 \item Data sharing is not applicable to this article as no datasets were generated or analysed during
the current study.\\

\item In case any datasets are generated during and/or analysed during the current study, they must
be available from the corresponding author on reasonable request.

\end{enumerate}

\end{document}